\documentclass[reqno,a4paper]{amsart}
\usepackage[latin1]{inputenc}
\usepackage[T1]{fontenc}
\usepackage[english]{babel}
\usepackage[allcolors=blue,colorlinks=true, pdfstartview=FitH, linkcolor=blue, pdftoolbar=true, bookmarks=true,bookmarksnumbered,plainpages,backref]{hyperref}
\usepackage[usenames,dvipsnames,svgnames,table]{xcolor}
\usepackage{graphicx}
\usepackage{caption}
\usepackage{subcaption}
\usepackage[normalem]{ulem}
\usepackage{bbold}
\usepackage{indentfirst}
\usepackage{tablefootnote}
\usepackage[lmargin=2.5cm,tmargin=2cm,bmargin=2cm,rmargin=2.5cm]{geometry}
\usepackage[usenames,dvipsnames,svgnames,table]{xcolor}
\usepackage{amsmath,amssymb,amsthm,amsfonts,esint}
\DeclareMathAlphabet{\mathcal}{OMS}{cmsy}{m}{n}
\usepackage{lipsum}

\usepackage{mathtools}
\usepackage[normalem]{ulem}

\usepackage{newtxtext}

\usepackage{enumitem}

\newtheoremstyle{theorem}
{6pt +1\p@ -2.0\p@}
{6pt +1\p@ -2.0\p@}
{\it}			      
{}				  
{\bfseries}   
{.}               
{.4em}       
{}               
\theoremstyle{theorem}
\newtheorem{theorem}{Theorem}[section]
\newtheorem{definition}[theorem]{Definition}
\newtheorem{proposition}[theorem]{Proposition}
\newtheorem{corollary}[theorem]{Corollary}
\newtheorem{lemma}[theorem]{Lemma}
\newtheorem{remark}[theorem]{Remark}
\newtheorem{example}[theorem]{Example}

\numberwithin{equation}{section}

\newcommand{\R}{\mathbb R}
\newcommand{\Rn}{{\R^n}}

\newcommand{\N}{\mathbb N}
\newcommand{\Z}{\mathbb Z}
\newcommand{\C}{\mathbb C}

\def\S{{\mathcal S}}

\newcommand{\bmo}{\hbox{\rm bmo}}

\newcommand{\lb}{\lbrace}
\newcommand{\rb}{\rbrace}

\newcommand{\Gd}{\delta}

\newcommand{\Gg}{\gamma}

\newcommand{\supp}{\text{supp}\,}

\DeclareMathOperator*{\BMO}{BMO}

\begin{document}
	
	\title{Inhomogeneous cancellation conditions and Calder\'on-Zygmund type operators on $h^p$}
	
	\author{Galia Dafni}
	\address{Department of Mathematics and Statistics, Concordia University, Montreal, QC, H3G 1M8, Canada}
	\email{galia.dafni@concordia.ca}
	
	\author{Chun Ho Lau}
	\address{Department of Mathematics and Statistics, Concordia University, Montreal, QC, H3G 1M8, Canada}
	\email{chunho.lau@concordia.ca}
	
	\author {Tiago Picon}
	\address{Departamento de Computa\c{c}\~ao e Matem\'atica, Universidade S\~ao Paulo, Ribeir\~ao Preto, SP, 14040-901, Brasil}
	\email{picon@ffclrp.usp.br}
	
	\author {Claudio Vasconcelos}
	\address{Departamento de Matem\'atica ,Universidade Federal de S\~ao Carlos, S\~ao Carlos, SP, 13565-905, Brasil}
	\email{claudio.vasconcelos@estudante.ufscar.br}
	
	\thanks{The first and second author were partially supported by the Natural Sciences and Engineering Research Council (NSERC) of Canada and the Centre de recherches math\'{e}matiques (CRM). The third author was supported by Conselho Nacional de Desenvolvimento Cient\'ifico e Tecnol\'ogico (CNPq - grant 311430/2018-0) and Fundaç\~ao de Amparo à Pesquisa do Estado de São Paulo (FAPESP - grant 18/15484-7). The fourth author was supported by Coordenação de Aperfeiçoamento de Pessoal de Nível Superior (CAPES), the Fonds de recherche du Québec Nature and technologies (FRQNT) and MITACS Globalink}
	
	
	
	\maketitle
	
	\begin{abstract}
		In this work we present a new approach to molecules on Goldberg's local Hardy spaces $h^p(\Rn)$, $0<p\leq1$, assuming an appropriate cancellation condition.  As applications, we prove a version of Hardy's inequality and improved continuity results for inhomogeneous Calder\'on-Zygmund operators on these spaces.
	\end{abstract}
	
	\section{Introduction}
	
	Atomic decomposition in the  Hardy spaces $H^p(\R^n)$, $0<p\leq 1$, allows us to write any tempered distribution  $f \in H^p(\R^n)$ as an infinite linear combination $\sum_{j}\lambda_{j}a_{j}$  of  atoms $a_{j}$, with $\|f\|_{H^p} \approx \inf \,(\sum |\lambda_{j}|^p)^{1/p}$ over all such decompositions.  An atom is a function supported in a ball (or a cube) which satisfies a size condition relative to that ball, and vanishing moment conditions. Several properties and applications of $H^p(\R^n)$ follow from this important tool; for instance, if $T: \mathcal{S}'(\R^n) \rightarrow \mathcal{S}'(\R^n)$ is a linear and continuous operator then its extension and continuity on $H^p(\R^n)$ can be established by just verifying that $\|Ta_{j}\|_{H^p} \leq C$ uniformly (see  \cite{Bownik2005, MedaSjogrenVallarino2008} for other cases). In the particular case when $T$ is a convolution-type singular integral operator, the uniform control on the atoms can be verified (see \cite[Chapter III p. 325]{GarciaFranciaWeighted}) and in addition, $M_{j}:=Ta_{j}$ are molecules.  In contrast to atoms, molecules do not require compact support, but satisfy a special concentration of Lebesgue norm associated to some ball, as well as vanishing moment conditions.  The molecular theory on Hardy spaces was first studied by Coifman  \cite{Coifman1974} in order to characterize the Fourier transform of distributions on $H^p (\R)$, and by Coifman, Taibleson and Weiss in the subsequent works \cite{Coifman-Weiss1977, Taibleson-Weiss}. This theory has been extensively explored in general settings, in particular to study the continuity of certain classes of non-convolution Calder\'{o}n-Zygmund operators and their generalizations (see \cite{AlvarezMilman}) on Hardy spaces, as well as on more general function spaces such as Triebel-Lizorkin (see \cite{RodolfoTorresBook}). 
	
	The spaces $H^p(\R^n)$ for $0<p\leq 1$ are not closed under multiplication by test functions, since this may destroy the global vanishing moment conditions. Consequently, we do not expect that linear operators such as non-convolution operators (for instance pseudodifferential operators) maps $H^p(\R^n)$ to itself. For this reason, Goldberg  \cite{Goldberg1979} introduced a localizable or inhomogeneous version of Hardy spaces, which he called {\em local Hardy spaces} and denoted by $h^p(\R^n)$, satisfying the continuous inclusions $H^p(\R^n) \hookrightarrow h^{p}(\R^n)$, $h^1(\R^n) \hookrightarrow L^{1}(\R^n)$, the equivalence $h^p(\R^n)=L^{p}(\R^n)$ for $p>1$ with comparable norms,  and the desired property: if $\varphi \in C_{c}^{\infty}(\R^{n})$ and $f \in h^{p}(\R^n)$ then $\varphi f \in h^{p}(\R^n)$. From the comparison between $H^{p}(\R^n)$ and $h^{p}(\R^n)$ (see \cite[Lemma 4]{Goldberg1979}), a natural atomic decomposition for $h^p(\R^n)$ arises which requires vanishing moment conditions only for atoms supported on balls $B$ with radius $r(B)\leq 1$; for atoms with $r(B)>1$, no moment conditions are required.  As an application,  Goldberg shows that pseudodifferential operators in the class $OpS^{0}_{1,0}(\R^n)$ are bounded on $h^{p}(\R^n)$ (see \cite{PiconKappHoepfner} for general symbols in H\"ormander classes).

	In contrast to $H^p (\R^n )$, the molecular theory for $h^p (\R^n )$, $0<p \leq 1$, is still not completely well understood. The initial formulation by Komori \cite{Komori2001}, for $n/(n+1) < p <1$, was used  to obtain boundedness of standard Calder\'on-Zygmund operators from $H^p(\R^n)$ to $h^p(\R^n)$. Recently, this result was extended by Ding, Han and Zhu \cite{DingHanZhu2020} to show the continuity of an inhomogeneous version of such operators, which includes pseudodifferential operators in the H\"ormander class $OpS_{1,0}^{0} (\R^n)$. Komori's molecular approach on $h^{p}(\R^{n})$ for $n/(n+1) < p <1$ assumes the uniform control of the moment given by 
	$$
	\left| \int{M(x)dx} \right| \leq \,C,
	$$ 
	which holds trivially for the case $p=1$ and is therefore not sufficient to characterize $h^1(\R^n)$ (see Example~\ref{example-atom}). Inhomogeneous cancellation conditions for $h^p(\R^n)$ atoms, $0<p\leq 1$, were previously introduced by the first author in \cite[Appendix B]{GaliaThesis}, where Goldberg's vanishing moment conditions on an atom $a$ associated to a ball $B=B(x_{B},r(B))$ with $r(B) < 1$ were relaxed to 
	$$
	\left| \int{a(x) (x-x_B)^{\alpha}dx} \right| \leq \,C \, r(B)^{ \, \eta} \ \text{for all} \ |\alpha| \leq \lfloor n(1/p-1) \rfloor \ \text{and some} \ \eta>0.
	$$
	For the case $p=1$, it was shown by the first author and Yue \cite[Definition 7.3]{GaliaYue} that the power condition can be weakened to a log-type one, i.e., there exists $C>0$ such that
	$$
	\left| \int{a(x)dx} \right| \leq \left[ \log \left( 1+\frac{C}{r(B)} \right) \right]^{-1}.
	$$
	This type of cancellation condition was further used by the first author and Liflyand \cite{GaliaLiflyand}  to give a molecular decomposition and prove Goldberg's version of Hardy's inequality (see \cite{Goldberg2}) for $h^1(\R)$. Approximate moments conditions have also been considered very recently by Ly and Naibo \cite{Naibo} in the Hermite setting.

	The aim of this work is establish a new atomic and molecular characterization of $h^p(\R^n)$, for all $0<p\leq  1$, and apply it to show improved continuity results for inhomogeneous Calder\'on-Zygmund operators  on $h^p(\R^n)$, as well as an $h^p(\R^n)$ version of Hardy's inequality.  The key is to introduce inhomogeneous cancellation conditions on both the operators and the atoms and molecules.  We extend Komori's approach for $p\leq n/n+1$, while reconciling it with  the log-type conditions in the case $p = 1$, by giving different cancellation properties when $p=n/(n+k)$  for $k \in \Z^{+}$, compared to $n/(n+k+1)<p<n/(n+k)$. 	 In fact, we show that a uniform cancellation condition of the type given by Komori is not sufficient for values of $p=n/(n+k)$ with $k \in \Z^+$ (see Example~\ref{example-atom}). The distinction between these two cases corresponds to the different definitions of the dual inhomogeneous Lipschitz spaces $\Lambda_{\gamma_p}(\R^{n})$ depending on whether $\gamma_p = n(1/p - 1)$ is an integer or not (e.g. for $p = 1$ we have $\bmo(\Rn)$ and for $p =\frac{n}{n+1}$, the Zygmund class). 
	
	The organization of the paper is as follows. We start in Section \ref{section-notation} by setting the notation and providing some background on function spaces. In Section \ref{section-approximate-atoms-molecules}, we present the atoms and molecules. 
	In Sections \ref{section-hardy-inequality}  we prove Goldberg's Hardy inequality for $h^{p}(\R^{n})$ for the full range $0<p\leq 1$, and  in \ref{section-applications} we extend the continuity results for inhomogeneous standard Calder\'on-Zygmund operators, studied by Ding et al \cite{DingHanZhu2020}, to a generalized class of operators on $h^p(\R^n)$, for all $0<p\leq 1$. In addition, we introduce the inhomogeneous strongly singular version of these operators and prove a similar continuity result.
	
	\section{Notations and Definitions} \label{section-notation}
	
	Here and in what follows, $\R^n$ denotes $n$-dimensional Euclidean space and $B$ denotes a generic ball in $\R^{n}$ in which $x_B$ and $r(B)$ will always stand for its center and radius, respectively. Given a locally integrable function $f$, we denote the mean of $f$ over $B$ by
	$$
	f_{B}\; := \fint_{B}f(x)dx \; := \frac{1}{|B|}\int_{B}f(x)dx \;,
	$$ 
	where $|B|$ is the Lebesgue measure of $B$. Given $0<p\leq 1$, we set
	$$
	\gamma_p := n\left( \frac{1}{p}-1 \right), \quad  N_p:= \lfloor \gamma_p \rfloor,
	$$
	where $\lfloor \cdot \rfloor$ is the floor function.
	
	\subsection{Lipschitz and generalized Campanato Spaces}  \label{section-lipschitz-campanato}
	
	\begin{definition}
		Let $0<\gamma<1$. A continuous function $f$ belongs to the homogeneous Lipschitz, also called H\"older, space
		$\dot{\Lambda}_{\gamma}(\Rn)$ if there exists $C>0$ such that 
		\[
		|f(x+h)-f(x)|\leq C|h|^{\gamma}
		\] 
		for every $x,h\in \R^{n}$. For $\gamma=1$,  $f$ belongs to $\dot{\Lambda}_1(\R^{n})$ (known as the Zygmund class)
		if there exists $C>0$ such that 
		\[
		|f(x+h)+f(x-h)-2f(x)|\le C|h|.
		\] 
		For $\gamma=k+\theta$, $k \in \N$, {$0< \theta \le 1$}, $f\in \dot{\Lambda}_{\gamma}(\R^{n})$ if all derivatives $\partial^\alpha f\in \dot{\Lambda}_{\theta}(\R^{n})$ for $\alpha\in\N^n$, $|\alpha|= k$. 
	\end{definition}
	\noindent We equip $\dot{\Lambda}_{\gamma}(\R^{n})$, $\gamma = k + \theta$, with the seminorm 
	\[
	|f|_{k+\theta}\,:=  \sum_{|\alpha|= k}
	\sup_{\begin{smallmatrix}x,h \in \R^{n}\\ h \ne 0\end{smallmatrix}} 
	\frac{|\partial^\alpha f(x+h)-\partial^\alpha f(x)|}{|h|^{\theta}} \quad \text{when } 0<\theta<1,
	\]
	or
	\[
	|f|_{k+1}\,:=  \sum_{|\alpha|= k}
	\sup_{\begin{smallmatrix}x,h \in \Rn\\ h \ne 0\end{smallmatrix}} 
	\frac{|\partial^\alpha f(x+h)+\partial^\alpha f(x-h)-2\partial^\alpha f(x)|}{|h|} \quad \text{when } \theta=1,
	\]
	and consider it modulo those functions for which $|f|_{\gamma}=0$, which are the polynomials of degree $\le \lfloor \gamma \rfloor$.
	
	Given $0<p<1$, the dual space of $h^{p}(\R^{n})$ is identified with the inhomogeneous Lipschitz space $\Lambda_{\gamma_p}(\R^{n})\;:=\;\dot{\Lambda}_{\gamma_p}(\R^{n}) \cap L^{\infty}(\R^{n})$, equipped with the norm {$\|f\|_{\Lambda_{\gamma_p}}:= |f|_{\gamma_p}+\|f\|_{L^\infty}$} (see \cite[Theorem 5]{Goldberg1979}). In the same paper, Goldberg also identified the dual of  $h^{1}(\R^{n})$  with the inhomogeneous version of the space BMO of functions of bounded mean oscillation.  This space, denoted $\bmo(\R^{n})$, consists of functions $f\in L^1_{loc}(\R^{n})$ such that
	\begin{equation}\label{bmo2}
		\|f\|_{bmo}:=\sup_{|B|< 1}\fint_{B}|f(x)-f_{B}|dx+ \sup_{|B|\geq 1}\fint_{B}|f(x)|dx<\infty.
	\end{equation}
	Equation \eqref{bmo2} defines a norm on $\bmo(\R^n)$ which is comparable to the one defined by Goldberg, in which balls $B$ are replaced by cubes $Q$ with sides parallel to the coordinate axes.	
	
	Next, we provide the definition of $\psi$-generalized Campanato spaces that will be useful for our applications in Section \ref{section-applications}.
	\begin{definition} \label{generalized-Campanato}
		Let $k\in \N\cup\lb 0\rb$, $1\leq q< \infty$, and $\psi:(0,\infty) \rightarrow (0,\infty)$. We define 
		\begin{equation} \label{space-gen-camp-1}
			L_{k}^{q,\psi}(\R^n):= \bigg\lb f\in L^{q}(\R^n): \ \exists\ C>0 \ s.t.\  \forall \ B\subset \R^n,\  \bigg( \fint_{B}|f(y)-(P^{k}_{B}f)(y)|^{q}dy\bigg)^{\frac{1}{q}}\leq C \, \psi(r(B)) \bigg\rb,
		\end{equation}
		where $P_{B}^{k}f(y)$ is the unique polynomial of degree less than or equal to $k$ that has the same moments as $f$ over $B$ up to order $k$. For the case $q=\infty$, with the same notation as above, we define
		\begin{equation} \label{space-gen-camp-2}
			L_{k}^{\infty,\psi}(\R^n):= \bigg\lb f\in L^{\infty}(\R^n): \ \exists\ C>0 \ s.t.\   \forall \ B,\   \Vert f(y)-(P^{k}_{B}f)\Vert_{L^{\infty}(B)}\leq C\psi(r(B)) \bigg\rb.
		\end{equation}
		The space $L_{k}^{q,\psi}(\R^n)$ is considered as a quotient space of the above classes of functions modulo all polynomials of degree less than or equal to $k$.
	\end{definition}
	Given $f\in L_{k}^{q,\psi}(\R^n)$,  we define
	$$\|f\|_{L_{k}^{q,\psi}} := \sup_{B} \frac{1}{\psi(r(B))}\bigg( \fint_{B}|f(y)-(P^{k}_{B}f)(y)|^{q}dy\bigg)^{\frac{1}{q}}$$
	if $1\leq q<\infty$, and
	$$\|f\|_{L_{k}^{\infty,\psi}} := \sup_{B} \frac{1}{\psi(r(B))}\|f(y)-(P^{k}_{B}f)(y)\|_{L^{\infty}(B)},$$
	where the supremum in both cases is taken over all balls in $\R^n$.
	
	There are several identifications of $\psi$-generalized Campanato spaces with other well known fuctions spaces in Harmonic analysis. In particular, if $k=0$, $1\leq q<\infty$ and $\psi\equiv1$, then $L^{q,\psi}_0(\R^n)\cong \BMO(\R^n)$, and if $\psi(t)=t^{\gamma}$, then $L^{q,\psi}_{\lfloor \gamma \rfloor}(\R^n)\cong \dot{\Lambda}_{\gamma}(\R^{n})$. In addition, following the proof of John-Nirenberg inequality (see \cite{SteinHarmonic} for instance) or the one for Morrey-Campanatos space in \cite{MR2379517}, we can see that if $\psi$ is an increasing function and $k\in \N \cup\lb 0\rb$, then for all $1\leq q<\infty$ we have $L_{k}^{q,\Psi}(\R^n) \cong L_{k}^{1,\Psi}(\R^n)$. We refer to \cite[Chapter III Section 5]{GarciaFranciaWeighted} for a detailed discussion on the relation between Campanato, Lipschitz and Zygmund spaces, and also to \cite{RafeiroCampanato} for an exposition of Campanato spaces on different domains and their generalizations.
	
	In what follows, given $0<p\leq 1$ and $\gamma_p \in \N$, we will consider the space $L_{\gamma_p}^{1,\Psi_p}(\R^n)$, where $\Psi_p(t) := t^{\, \Gg_p} \,	\log \left( 1+\dfrac{1}{\omega t} \right)^{-\frac{1}{p}}$ for some $\omega \geq 0$. 
	
	\section{Approximate atoms and molecules on $h^p$} \label{section-approximate-atoms-molecules}
	
	The spaces $h^p(\Rn)$, $p>0$,  introduced by  Goldberg \cite{Goldberg1979}, can be given by several definitions. We choose to use the maximal function definition.  Fix $\phi \in \S(\Rn)$ such that $\int{\phi(x)dx} \neq 0$, and set $\phi_t(x) = t^{-n}\phi(t^{-1}x)$. We say that a tempered distribution $f$ lies in $h^p(\Rn)$ if
$$\|f\|_{h^p(\Rn)} := \|m_{\phi}\|_{L^p(\Rn)} < \infty, \quad	m_{\phi}f(x) := \sup_{0<t<1} \left| f\ast \phi_t (x) \right|.$$
This class of distributions remains the same no matter which $\phi$ we choose, but  $\|f\|_{h^p}$ may vary.  For $p \geq 1$ this defines a norm, $h^p(\Rn)$ is a Banach space and when $p > 1$ it is equal to $L^p(\Rn)$ with equivalent norms.  We have continuous embeddings $\S(\Rn) \subset h^1(\Rn) \subsetneq L^1(\Rn)$ and so the latter containment is dense.  For $0<p\le1$, the space $h^p(\Rn)$ is a complete metric space with the distance $d(f,g)=\|f-g\|_{h^p}^p$.   Although  $h^p(\Rn)$ is not locally convex for $0<p<1$ and $\|f\|_{h^p}$ is on only quasi-norm \cite{Tri}, we will still refer to the functional $\|f\|_{h^p}$ as the $h^p$ {\em norm} for simplicity. The homogeneous Hardy spaces $H^p(\Rn)$ is contained in the corresponding $h^p(\Rn)$ since the maximal function defining $H^p$ can be taken as $M_\phi$, defined as $m_\phi$ above but allowing $0 < t < \infty$, and this inclusion is proper.

	Analogously to the case of $H^p(\Rn)$, an atomic decomposition can be obtained for $h^p(\Rn)$. Next we present the definition of $L^s$ atoms for $h^p$ (see \cite{Goldberg1979} for the case $s=\infty$), which we will denote as $(p, s)$ atoms for the sake of simplicity, even though this notation is also used in the literature (see for example \cite{GarciaFranciaWeighted}) for $L^s$ atoms for $H^p$.  When the atom is supported in a small ball, there is no difference, but when it is supported in a large ball, the moment conditions from the $H^p$ case are no longer required.
		
	\begin{definition} \label{goldberg-atom}
		Let $0<p\le 1 \le s \le \infty$ with $p<s$. A measurable function $a$ is called a $(p,s)$ atom (for $h^p(\Rn)$) if there exists a ball $B\subset \Rn$ such that
		\begin{equation*}
			\textnormal{(i)} \ supp(a) \subset B; \quad \textnormal{(ii)} \ \| a \|_{L^s} \le r(B)^{n(\frac{1}{s}-\frac{1}{p})}; \quad \textnormal{(iii)} \ \textnormal{if} \ r(B)<1, \int{a(x)x^{\alpha}}dx=0 \ \textnormal{for all} \ |\alpha|\le N_p \ . 
		\end{equation*}
	\end{definition}
	Note that  conditions (i) and (ii) alone give
 	\begin{equation} 
	\label{eq-large_atom} 
	\bigg|  \int a(x)(x-x_{B})^{\alpha}dx  \bigg| \le r(B)^{|\alpha|} \|a\|_{L^{s}} |B|^{1-\frac{1}{s}} \lesssim \ r(B)^{|\alpha|-\gamma_p}. 
  	\end{equation} 
	Here and in what follows, the notation $f \lesssim g$ means that there exists a constant $C>0$ such that $f\le C g$.
  Thus for a $(p,s)$ atom supported in a ball $B$ with $r(B) \ge 1$, the $\alpha$-th moment, $|\alpha| \le N_p$, is bounded by a constant depending only on $p, s$ and $n$.
  
The atomic decomposition (see \cite{Goldberg1979} for the strongest case, $s = \infty$) asserts that for $f \in h^p(\Rn)$, there exists a sequence $\{a_j\}_{j \in \N}$ of $(p,s)$ atoms in $h^p(\Rn)$ and $\{\lambda_j\}_{j \in \N}$ complex scalars in $\ell^p(\C)$ such that 	
	\begin{equation} \label{goldberg-decomp}
		f = \sum_{j \in \N} \lambda_j \, a_j \ \ \mbox{in the sense of distributions, and} \ \ \| f \|_{h^p} \approx \inf  \left( \sum_{j \in \N} |\lambda_j|^p \right)^{\frac{1}{p}},  
	\end{equation}
	where the infimum is taken over all such atomic representations.  The decomposition also converges in the $h^p(\Rn)$ norm.

	\subsection{Approximate atoms} \label{subsectio-approx-atoms}
	
	The spaces $h^p(\R^n)$, $p>0$,  introduced by  Goldberg \cite{Goldberg1979}, can be given by several definitions. We choose to use the maximal function definition.  Fix $\phi \in \S(\R^n)$ such that $\int{\phi(x)dx} \neq 0$, and set $\phi_t(x) = t^{-n}\phi(t^{-1}x)$. We say that a tempered distribution $f$ lies in $h^p(\R^n)$ if
	$$	
	\|f\|_{h^p(\Rn)} := \|m_{\phi}\|_{L^p(\Rn)} < \infty, \quad	m_{\phi}f(x) := \sup_{0<t<1} \left| f\ast \phi_t (x) \right|.
	$$
	This class of distributions remains the same no matter which $\phi$ we choose, but  $\|f\|_{h^p}$ may vary.  For $p \geq 1$ this defines a norm, $h^p(\R^n)$ is a Banach space and when $p > 1$ it is equal to $L^p(\R^n)$ with equivalent norms.  We have continuous embeddings $\S(\Rn) \subset h^1(\R^n) \subsetneq L^1(\R^n)$ and so the latter containment is dense.  For $0<p\le1$, the space $h^p(\R^n)$ is a complete metric space with the distance $d(f,g)=\|f-g\|_{h^p}^p$.   Although  $h^p(\R^n)$ is not locally convex for $0<p<1$ and $\|f\|_{h^p}$ is on only quasi-norm \cite{Tri}, we will still refer to the functional $\|f\|_{h^p}$ as the $h^p$ {\em norm} for simplicity. The homogeneous Hardy spaces $H^p(\R^n)$ is contained in the corresponding $h^p(\R^n)$ since the maximal function defining $H^p$ can be taken as $M_\phi$, defined as $m_\phi$ above but allowing $0 < t < \infty$, and this inclusion is proper.
	
	Analogously to the case of $H^p(\R^n)$, an atomic decomposition can be obtained for $h^p(\R^n)$. Next we present the definition of $L^s$ atoms for $h^p$ (see \cite{Goldberg1979} for the case $s=\infty$), which we will denote as $(p, s)$ atoms for the sake of simplicity, even though this notation is also used in the literature (see for example \cite{GarciaFranciaWeighted}) for $L^s$ atoms for $H^p$.  When the atom is supported in a small ball, there is no difference, but when it is supported in a large ball, the moment conditions from the $H^p$ case are no longer required.
	
	\begin{definition} \label{goldberg-atom}
		Let $0<p\le 1 \le s \le \infty$ with $p<s$. A measurable function $a$ is called a $(p,s)$ atom (for $h^p(\R^n)$) if there exists a ball $B\subset \R^n$ such that
		\begin{equation*}
			\textnormal{(i)} \ supp(a) \subset B; \quad \textnormal{(ii)} \ \| a \|_{L^s} \le r(B)^{n\left(\frac{1}{s}-\frac{1}{p}\right)}; \quad \textnormal{(iii)} \ \textnormal{if} \ r(B)<1, \int{a(x)x^{\alpha}}dx=0 \ \textnormal{for all} \ |\alpha|\le N_p \ . 
		\end{equation*}
	\end{definition}
	Note that  conditions (i) and (ii) alone give
	\begin{equation} 
		\label{eq-large_atom} 
		\bigg|  \int a(x)(x-x_{B})^{\alpha}dx  \bigg| \le r(B)^{|\alpha|} \|a\|_{L^{s}} |B|^{1-\frac{1}{s}} \lesssim \ r(B)^{|\alpha|-\gamma_p}. 
	\end{equation} 
	Here and in what follows, the notation $f \lesssim g$ means that there exists a constant $C>0$ such that $f\le C g$.
	Thus for a $(p,s)$ atom supported in a ball $B$ with $r(B) \ge 1$, the $\alpha$-th moment, $|\alpha| \le N_p$, is bounded by a constant depending only on $p, s$ and $n$.
	
	The atomic decomposition (see \cite{Goldberg1979} for the strongest case, $s = \infty$) asserts that for $f \in h^p(\R^n)$, there exists a sequence $\{a_j\}_{j \in \N}$ of $(p,s)$ atoms in $h^p(\R^n)$ and $\{\lambda_j\}_{j \in \N}$ complex scalars in $\ell^p(\C)$ such that 	
	\begin{equation} \label{goldberg-decomp}
		f = \sum_{j \in \N} \lambda_j \, a_j \ \ \mbox{in the sense of distributions, and} \ \ \| f \|_{h^p} \approx \inf  \left( \sum_{j \in \N} |\lambda_j|^p \right)^{\frac{1}{p}},  
	\end{equation}
	where the infimum is taken over all such atomic representations.  The decomposition also converges in the $h^p(\R^n)$ norm.

	\subsection{Approximate atoms} \label{subsectio-approx-atoms}
	
	In this section we define a notion of  atoms which does not distinguish between small and large values $r(B)$ as in Definition \ref{goldberg-atom}. Instead, our cancellation condition is intrinsically related to the value of $0<p\le1$:  if $p \neq \frac{n}{n+k}$ for every $k \in \Z_{+}$, namely $N_p < \gamma_p$, it suffices to bound the moments up to order $|\alpha|\le N_p$ by a constant; on the other hand, if $p=\frac{n}{n+k}$, a $\log$-type control is needed when $|\alpha|=N_p = \gamma_p$.   	
	\begin{definition} 
		\label{approximate-atom}
		Let $0<p\le 1 \le s \le \infty$ with $p<s$,  $\omega \ge 0$, and define $\varphi_p: (0,\infty) \rightarrow (0,\infty)$ by
		$$		\varphi_p(t):= \left[\log \left( 1+\dfrac{1}{\omega t} \right)\right]^{-\frac{1}{p}},
		$$
		where $\varphi_p(t)=0$ in the limiting case $\omega=0$. We say that a measurable function $a$ is a $(p,s,\omega)$ atom (for $h^p(\Rn)$) if there exists a ball $B\subset \R^n$ such that 
		$$
		\textnormal{(i)} \ supp(a) \subset B; \quad \textnormal{(ii)} \ \| a \|_{L^s} \leq r(B)^{n(\frac{1}{s}-\frac{1}{p})};$$
		and
		$$\textnormal{(iii)'} \ \ \ \ \left\{ \begin{array}{ll} \displaystyle \left| \int_{B}{a(x) (x-x_B)^{\alpha}dx} \right| \le \omega, &\quad \text{if } |\alpha|<\gamma_p, \\ 
			& \\
			\displaystyle \left| \int_{B}{a(x) (x-x_B)^{\alpha}dx} \right| \le \varphi_p(r(B)) &\quad \text{if }|\alpha|=N_p = \gamma_p.
		\end{array} \right.
		$$
	\end{definition}
	This definition covers the one in \cite[Definition 7.3]{GaliaYue} for the case $p=1$, the one in \cite[Lemma 3]{Komori2001} for the case $n/(n+1)<p<1$, and when $\omega = 0$, the $(p,s)$ atoms for the homogeneous Hardy space $H^p$.
	
	Combining \eqref{eq-large_atom} for $r(B) \ge 1$ with the fact that $\varphi_p(t) \le  [\log ( 1+\omega^{-1} )]^{-\frac{1}{p}}$ for $t < 1$, we have that in all cases, the moments of a $(p, s, \omega)$ atom satisfy
	\begin{equation}
		\label{constantbound}
		\left| \int_{B}{a(x) (x-x_B)^{\alpha}dx} \right|  \le C_{p,\omega} \quad \text{for }|\alpha|\le N_p.
	\end{equation}

	Note that  $(p,s)$ atoms supported in small balls are automatically $(p,s, \omega)$ atoms, and if $\omega > 0$, $(p,s)$ atoms  supported in large balls are $(p,s, \omega)$ atoms up to multiplication by a constant depending on $\omega, n, p$ and $s$, thanks to \eqref{eq-large_atom}.	Thus Goldberg's atomic decomposition~\eqref{goldberg-decomp} gives us a decomposition of any $f \in h^p(\R^n)$ into $(p,s,\omega)$ atoms.

	Conversely, to show that every infinite linear combination of $(p, s,\omega)$ atoms, $\sum_{j \in \N} \lambda_j a_j$, taken in the sense of distributions, is in $h^p(\R^n)$ with norm bounded by a constant times $\{\lambda_j\}_{\ell^p}$,  it is enough to show the following bound on the $h^p$ norm of a single $(p, s,\omega)$ atom.  The use of the maximal function norm and the convergence in $\mathcal{S}'(\Rn)$ allows us to pass from that estimate to the norm of the sum.	
	
	\begin{proposition} \label{prop-psomega-atom}
		If $a$ is a $(p,s, \omega)$ atom, then $\| a \|_{h^p} \lesssim 1$, where the constant depends on $p, s$ and $\omega$.
	\end{proposition}
	
	\begin{proof}
		Let $a$ be a $(p,s, \omega)$ atom supported in $B=B(x_B,r)$. Split
		$$
		\Vert m_{\phi} a \Vert^{p}_{L^p} = \int_{2B} \left(\sup_{0<t<1} |\phi_t\ast a(x)|\right)^p dx + \int_{(2B)^c} \left(\sup_{0<t<1} |\phi_t\ast a(x)|\right)^p dx.	
		$$
		Using the fact that $\displaystyle \sup_{0<t<1}|\phi_t\ast a(x)|\le C_{\phi}\, Ma(x)$, where $M$ denotes the Hardy-Littlewood maximal function, which is bounded from $L^{s}(\R^{n})$ to itself for $1<s\le \infty$, it follows that 
		$$
		\int_{2B} \bigg(\sup_{0<t<1} |\phi_t\ast a(x)|\bigg)^p dx \le C_{\phi,s} |2B|^{1-\frac{p}{s}} \Vert a \Vert_{L^s}^p \le
		C_{\phi,s,p,n} \,\, r^{n\left(1-\frac{p}{s}\right)} r^{n \left(\frac{p}{s}-1\right)} = C_{\phi,s,p,n}.
		$$
		Note the last estimate holds for all $0<p\le1$. For $s=1$ and $p<1$, using that $M$ is of weak type $(1,1)$, the same estimate follows (see \cite[Lemma 3.1 p. 248]{GarciaFranciaWeighted}).
		
		Now we deal with the estimate outside $2B$. From the Taylor expansion of the function $y \mapsto \phi_t(x-y)$ up to order $N_p$, we may write
		\begin{align*}
			\sup_{0<t<1} |\phi_t\ast a(x)| &= \sup_{0<t<1}  \bigg|\int_{\R^n} \sum_{|\alpha| \le N_p-1} C_{\alpha} \, \partial^{\alpha}\phi_t(x-x_B)\; (x_B-y)^{\alpha} a(y)dy \\
			& \quad \quad + \int_{\R^n }\sum_{|\alpha|= N_p} C_{\alpha} \, \partial^{\alpha}\phi_t(x-x_B+c(x_B-y))\; (x_B-y)^{\alpha}  a(y)dy \bigg| \\
		\end{align*}
		for some $c \in \, (0,1)$.  As $|x-x_B|\ge 2r$ and $|y-x_B|\le r$, we have$|x-x_B+c(x_B-y)|\ge |x-x_B|/2$. 
		
		For $\phi \in \mathcal{S}(\R^n)$, we will use the bound
		$
		\vert \partial^{\alpha}\phi(x) \vert \le C_{\alpha} \vert x \vert ^{-N},
		$ 
		where $N>0$, depending on $|\alpha|$, will be chosen conveniently.  Breaking the integral into the integrals over the regions $2r<|x-x_B|\le 2$ (empty if $r \ge 1$) and $|x-x_B|\ge 2$, we take $N=n+|\alpha|$ for the first region and $N=n+N_p+1$ for the second one. Since the supremum in $t$ is taken over $(0,1)$, we have $t^{-n-|\alpha|+n+N_p+1}\le 1$ for all $|\alpha| \le N_p$. Thus
		\begin{align*}
			&\int_{(2B)^c}(\sup_{0<t<1} |\phi_t\ast a(x)|)^p dx  \\
			&\le \int_{2r<|x-x_B|\le 2} \bigg(\sup_{0<t<1} \sum_{|\alpha| \le N_p} C_{\alpha} \, t^{-n-|\alpha|} \bigg|\frac{x-x_B}{t} \bigg|^{-n-|\alpha|}  \bigg|\int_{\R^n}a(y)(y-x_B)^{\alpha} dy\bigg|\bigg)^pdx \\
			&+\int_{|x-x_B|\ge 2}\bigg(\sup_{0<t<1} \sum_{|\alpha| \le N_p} C_{\alpha} \, t^{-n-|\alpha|} \bigg|\frac{x-x_B}{t} \bigg|^{-n-N_p - 1}  \bigg|\int_{\R^n}a(y)(y-x_B)^{\alpha} dy\bigg|\bigg)^pdx\\
			&\lesssim \sum_{|\alpha| \le N_p} \,  \bigg|\int_{\R^n}a(y)(y-x_B)^{\alpha} dy\bigg|^p\left(\int_{2r<|x-x_B|\le 2} |x-x_B|^{-np-|\alpha|p} dx +  \int_{|x-x_B|\ge 2} |x-x_B|^{-p(n+N_p+1)} dx \right).\\
			&\lesssim \sum_{|\alpha| \le N_p} \,  \bigg|\int_{\R^n}a(y)(y-x_B)^{\alpha} dy\bigg|^p\int_{2r<|x-x_B|\le 2} |x-x_B|^{-np-|\alpha|p} dx +  C_{n,p,\omega}.\\
		\end{align*}
		Here we have used \eqref{constantbound} and the fact that $p(n+N_p+1) > n$ to bound the terms involving the integral over $|x-x_B|\ge 2$. 
		
		The other terms are nonzero only when $r < 1$. In the case $p \ne n/(n+k)$ for any $k \in \Z_{+}$, meaning $N_p<\gamma_{p}$, we have $-np-|\alpha|p > -n$ for all $|\alpha| \le N_p$ so the integral over $|x-x_B| \le 2$ is convergent, and together with condition (iii)', this gives a bound which is a constant multiple of $\omega$.
		
		The same bound also works when $p = n/(n+k)$, $k \in \Z_{+}$, but $|\alpha| < N_p$.  When $|\alpha| = N_p = \gamma_p$, we have $-np-|\alpha|p = -n$ and therefore $\int_{2r<|x-x_B|\le 2} |x-x_B|^{-np-|\alpha|p} dx \approx \log r$.  Using condition (iii)' again, this time with the log bound on the moments, gives a multiple of  $\log r \; \varphi_p(r)^p$, which is bounded for $r \le 1$.
		
	\end{proof}
	
	Next we show that in Definition \ref{approximate-atom}, assuming the support and size conditions (i) and (ii), we cannot replace (iii)' by a uniform bound on the moments of $a$ when $p=n/(n+k)$ for $k \in \Z_{+}$.  That is, an estimate of the form \eqref{constantbound} is not sufficient to conclude that $a \in h^{p}(\R^n)$, {and in fact the logarithmic  (in $1/r(B))$) decay on the highest moments in (iii)' is necessary in this case}. The construction of the counterexample  is based on \cite[Lemma 3.1]{PiconKappHoepfner}.
	
	\begin{example} \label{example-atom} \textnormal{
	Let $n=1$ and $p=1/k$ for some $k \in \N$ (hence $N_p=k-1$). Take $\varphi:\R_+ \mapsto \R_+$ to be any bounded function with $\|\varphi\|_\infty \le 1$.  For $0 < r < 2^{1 - k}$, we will construct a function $a$, depending on $r$, and satisfying conditions (i) and (ii) of Definition~\ref{goldberg-atom} for a $(p,\infty)$ atom with respect to the interval $B =  [-2^{k-2}r, \,2^{k-2}r]$.  Moreover, we will show
			 \begin{equation} \label{integral-1}
				\int a(t)t^{\ell}dt=0 \ \ \text{for all } 0\le \ell \le k-2 \quad \mbox{and}\quad \int a(t)t^{k-1}dt=C_k\varphi(r)
			\end{equation}
			for a positive constant $C_k$ independent of $r$.  Finally, we will see that $\varphi(r) |\log r| \lesssim \|a\|_{h^p}$.  As a result, letting $r$ tend to  $0$, we conclude that the norm of $a$ can remain bounded only if $\varphi(r) = {\mathcal O} (1/|\log r|)$.}
			
			\textnormal{We start by defining the odd function
			$$
			a_1(t) = \left\{\begin{array}{ll}
				\left( 2^{k-1}r\right)^{-k}\varphi(r) & \text{if } t \in [0,r]; \\
				-\left( 2^{k-1}r\right)^{-k}\varphi(r)  & \text{if } t \in [-r,0); \\
				0 & \text{if } |t|>r.
			\end{array}\right.
			$$
			Now we construct $a_2$ by translating  the previous function $a_1$ by $r$ units to the right half-line and then extending it to $[-2r,0]$ in such a way that the resulting function is even. We proceed in the same way, translating $a_2$ by $2r$ and constructing an odd function $a_3$, and so on. Inductively, define
			$$
			a_{m+1}(t) = \left\{\begin{array}{l}
				a_{m}(t-2^{m-1}r)-a_m(-t-2^{m-1}r), \quad \text{if } m \text{ is even;} \\
				a_{m}(t-2^{m-1}r)+a_m(-t-2^{m-1}r), \quad \text{if } m \text{ is odd.}
			\end{array}\right.
			$$
			Observe that $a_m$ is an even function if $m$ is even and it is an odd function if $m$ is odd. In addition,
			$$
			\supp(a_m) \subset [-2^{m-1}r, \,2^{m-1}r] \ \ \text{and} \ \ |a_m(t)| \le (2^{k-1}r)^{-k}\|\varphi\|_\infty \le  \left| [-2^{k-2}r, \,2^{k-2}r] \right|^{-k}.
			$$
			When $m = k-1$, this shows that the function $a = a_{k-1}$ satisfies  conditions (i) and (ii)  of Definition~\ref{approximate-atom}. }
			
			\textnormal{We want to show that \eqref{integral-1} holds for $a = a_{k-1}$.
			The first  identity in \eqref{integral-1} follows from \cite[Lemma 3.1]{PiconKappHoepfner}. For the sake of completeness, we are going to show both by proving the following identities for any $m$:
			\begin{equation} \label{integral-2}
				\int a_{m}(t)t^{\ell}dt=0 \ \ \text{for all } 0\le \ell \le m-1 \quad \mbox{and}\quad \int a_{m}(t)t^{m}dt=C_{m,k}r^{m+1-k},
			\end{equation}	
			where $C_{m,k}$ is a positive constant  independent of $r$.
			We proceed  to prove the first identity in \eqref{integral-2}, the vanishing moments, by induction on $m$. We have $\int a_1(t) dt = 0$ by the oddness of $a_1$.  Now assuming the vanishing moments for $a_{m}$, we show it for $a_{m+1}$.  Suppose without loss of generality that $m$ is odd (the same argument works if $m$ is even). By construction, $a_{m+1}$ is even and the vanishing moments will immediately hold for every odd $\ell$, $1\le l \le m$. Suppose $\ell$ is even.  Then using the definition of $a_{m+1}$ and the support of $a_m$ we have
			\begin{align*}
				\int a_{m+1}(t)t^{\ell}dt &= 2\int_{0}^{2^{m}r}a_{m+1}(t)t^{\ell}dt \\
				&=2\int_{0}^{2^{m}r}\left[a_{m}(t-2^{m-1}r)+a_m(-t-2^{m-1}r)\right] t^{\ell}dt \\
				&=2\int_{0}^{2^{m}r}a_{m}(t-2^{m-1}r)t^{\ell}dt  \\
				&=2\int_{-2^{m-1}r}^{2^{m-1}r} a_{m}(t)(t+2^{m-1}r)^{\ell}dt \\
				&= \sum_{\gamma\le \ell-1<m} C_{\ell, \gamma, r, m} \int_{-2^{m-1}r}^{2^{m-1}r} a_{m}(t)t^{\gamma}dt \\
				&= 0.
			\end{align*}
			In the last two steps we have used the fact that $a_{m}(t)t^\ell$ is odd to eliminate the integral of the highest order term in the binomial expansion, followed by the induction hypothesis.}
		
		\textnormal{For the second identity in \eqref{integral-2}, we will follow the same procedure.  First we compute 
			$$\int a_1(t) t dt = \left( 2^{k-1}r\right)^{-k} \varphi(r) 2\int_0^{r} t dt =  \left( 2^{k-1}\right)^{-k} \varphi(r) r^{2 - k}.$$
			Assuming that  $\int a_{m}(t)t^{m}dt=C_{m,k} \varphi(r) r^{m+1-k}$ for some positive constant $C_{m,k}$ independent of $r$, we write, as above,
			\begin{align*}
				\int a_{m+1}(t)t^{m+1}dt &= 2\int_{0}^{2^{m}r}a_{m+1}(t)t^{m+1}dt \\
				&=2\int_{-2^{m-1}r}^{2^{m-1}r} a_{m}(t)(t+2^{m-1}r)^{m+1}dt \\
				&=2\int_{-2^{m-1}r}^{2^{m-1}r} a_{m}(t)(m+1)t^m\; 2^{m-1}r\; dt \\
				&=(m+1)2^m r \int a_{m}(t) t^m dt \\
				&=(m+1)2^m C_{m,k} \varphi(r) r^{m+2-k}.
			\end{align*}
			Here again we have used the fact that $a_{m}(t)t^{n+1}$ is odd to eliminate the integral of the highest order term in the binomial expansion, as well as the vanishing moments of $a_m$ of order $\ell$ for all $\ell < m$, followed by the induction hypothesis.  This proves the induction step with $C_{m+1,k} = (m+1)2^m C_{m,k}$.}

		\textnormal{Now that we have shown that $a = a_{k-1}$ satisfies the conditions of Definition~\ref{approximate-atom} for a $(p,\infty)$ atom with the bound on the highest-order moment in (iii)' replaced by $C_k \varphi(r)$, we want to estimate its $h^p$ norm.  We will do this by testing against an element $f$ of the dual space $\Lambda^{k-1}(\R) = (h^p(\R))^{\ast}$.  
			Fix a cutoff function $\eta \in C^\infty$ with support in $(-1,1)$ which is equal to $1$ on $[-1/2,1/2]$, and let $f$  be given by
			$$
			f(t) = t^{k-1}\log|t|\eta(t).	
			$$
			Recall that we are assuming $r < 2^{1 - k}$, so that $\eta = 1$ on the support of $a$, and we have, by \eqref{integral-1},
			\begin{align*}
				\bigg| \int a(t) f(t)dt - C_k \varphi(r)\log r \bigg| &= \left| \int a(t) t^{k-1} (\log|t| - \log r) dt \right| \\
				&= \left| 2 \int_{0}^{2^{k-2}r} a(t)t^{k-1}\log(t/r)dt \right|  \\
				&= \left| 2r^{k} \int_{0}^{2^{k-2}} a(tr)t^{k-1}\log t dt \right| \\
				&\le  2r^{k} \|a\|_\infty\int_{0}^{2^{k-2}} t^{k-1} |\log(t)|dt   \\
				&\le \widetilde{C}_k,
			\end{align*}
			where $\widetilde{C}_k$ is independent of $r$.
			This shows 
			$$ \varphi(r)|\log r| \le C_k^{-1}\left(\widetilde{C}_k + \left|\int a(t)f(t)dt \right|\right)  \lesssim 1  + \|f\|_{\Lambda^{k-1}}\|a\|_{h^p} \lesssim 1 + \|a\|_{h^p}$$
			with constants depending on $k$ but independent of $r$.}
	\end{example}
	
	\subsection{Approximate Molecules} \label{subsection-approx-molecules}
	
	In this section, we present a new class of molecules analogous to the $(p,q, \omega)$ atoms defined above, and a molecular decomposition on $h^{p}(\R^{n})$ for the full range $0<p\leq 1$.  Note that the notation $(p,q,b)$-molecule was used in \cite{GarciaFranciaWeighted} for $H^p$ molecules with full cancellation conditions.
	
	\begin{definition} \label{hpmolecule}
		Let $0<p\leq 1 \leq s< \infty$ with $p<s$, $\lambda > n \left(\frac{s}{p}-1\right)$, and $\omega$, $\varphi_p$ be as in Definition~\ref{approximate-atom}. We say that a measurable function $M$ is a 
		$(p,s,\lambda, \omega)$ molecule (for $h^p(\Rn)$) if there exists a ball $B \subset \Rn$ and a constant $C>0$ such that
		\begin{itemize}	      \item[\textnormal{M1.}] $\|M\|_{L^s(B)} \leq C \, r(B)^{n\left(\frac{1}{s}-\frac{1}{p}\right)}$
			\item[\textnormal{M2.}] $\displaystyle \left( \int_{B^c}{|M(x)|^s \, |x-x_B|^{\lambda}dx}\right)^{1/s} \leq C \, r(B)^{\frac{\lambda}s + n \left(\frac{1}{s}-\frac{1}{p}\right)}$
			\item[\textnormal{M3.}] $\displaystyle \left\{ \begin{array}{ll} \displaystyle \left| \int_{\Rn}{M(x) (x-x_B)^{\alpha}dx} \right| \leq \omega, &\quad \text{if } |\alpha|<\gamma_p, \\ 
				& \\
				\displaystyle \left| \int_{\Rn}{M(x) (x-x_B)^{\alpha}dx} \right| \leq \varphi_p(r(B)) &\quad \text{if } |\alpha|=N_p =\gamma_p.
			\end{array} \right.$
		\end{itemize}
		We call the molecule ``normalized" if $C = 1$.
	\end{definition}
	Choosing $s=1$, the previous definition covers the molecules introduced by Komori in \cite{Komori2001}[Definition 4.4] for $n/(n+1) < p < 1$. In particular, our definition not only extends it to values of $p \leq \frac{n}{n+1}$ but also provides an appropriate bound for the size of the moment condition when $p=n/(n+k)$ for $k \in \Z^{+}$.

	\begin{remark} \label{remark-molecules} \textcolor{white}{.}
	\begin{itemize}
		\item[\textnormal{(i)}] \textnormal{Conditions (M1) and (M2) together are equivalent to the pair of global estimates (with a different choice of the constant $C$)
		\begin{equation*}
			\|M\|_{L^s(\Rn)}  \leq C \,  r(B)^{n\left(\frac{1}{s}-\frac{1}{p}\right)}\,\,\,\text{ and } \,\,\, \|M|\cdot-x_B|^{\frac{\lambda}{s}}\|_{L^s(\Rn)}  \leq C \, r(B)^{\frac{\lambda}s + n \left(\frac{1}{s}-\frac{1}{p}\right)}.
		\end{equation*}
		Moreover, if (M2) holds for a given $\lambda$, then the analogous estimate holds for any $\lambda' < \lambda$:
		\begin{equation*}
			\int_{B^c}{|M(x)|^s \, |x-x_B|^{\lambda'}dx} \leq r(B)^{\lambda' - \lambda} \int_{B^c}{|M(x)|^s \, |x-x_B|^{\lambda}dx} \le  r(B)^{\lambda' + n \left(1-\frac{s}{p}\right)}.
		\end{equation*}
		so combined with (M1) we also have the corresponding bound on $ \|M|\, \cdot-x_B|^{\frac{\lambda'}{s}}\|_{L^s(\Rn)}$. }
		\item[\textnormal{(ii)}] \textnormal{As in \eqref{eq-large_atom}, assuming conditions (M1) and (M2) only, we can derive the following estimate on the moments of $M$. Let $A_j := \{ x \in \Rn: \ 2^{j}r\le |x-x_B|<2^{j+1}r  \}$ with $r=r(B)$, $j \geq 0$.  Then for $|\alpha| \leq \gamma_p \leq  n \left(\frac{1}{p}-\frac{1}{s}\right) <\frac{\lambda}s$ we have
		\begin{align*}
			\left| \int M(x) (x-x_B)^{\alpha}dx \right| 
			&\lesssim r^{|\alpha|}|B|^{1-\frac{1}{s}}\|M\|_{L^s(B)} +\sum_{j=0}^{\infty}(2^jr)^{|\alpha|-\frac{\lambda}{s}}|A_j|^{1-\frac{1}{s}} \left( \int_{B^c}{|M(x)|^s \, |x-x_B|^{\lambda}dx} \right)^{\frac{1}{s}}  \\
			&\lesssim \, r^{|\alpha|  - \gamma_p} 
		\end{align*}
		This shows that (M3) holds automatically, with some constant $C_{n,s}$ in place $\omega$, for all balls with $r(B) \geq 1$. 	On the other hand, for small balls, condition (M3) is strongest when $\alpha = 0$.}
		\end{itemize}
		\end{remark}
	
	Next we show that we can decompose a $(p,s,\lambda, \omega)$ molecule into approximate atoms with uniformly bounded norm in $h^p(\Rn)$ .
	
	\begin{proposition} \label{molecular-decomp}
		If $M$ is a normalized $(p,s,\lambda, \omega)$-molecule, then  $\| M \|_{h^p} \lesssim 1$, with the constant depending on the parameters $p,n,s,\lambda, \omega$ but not on $M$.
	\end{proposition}
	
	\begin{proof}
		The proof is inspired by the classical molecular decomposition for real Hardy spaces. We will outline only the main ideas, highlighting the parts that diverges from the classical proof, which can be consulted in \cite[Theorem III.7.16]{GarciaFranciaWeighted}, preserving some of the notation from that proof below. Let $M$ be a normalized $(p,s,\lambda, \omega)$ molecule associated to a ball $B=B(x_B,r)\subset \R^n$.  We show that
		$$
		M = \sum_{k=0}^\infty{t_{k} \, a_k} + \sum_{k=0}^\infty{s_{k} \, b_k} + a_\omega,
		$$
		where $a_k$, $b_k$ are $(p,s)$ and $(p,\infty)$ atoms with full cancellation, respectively, and $a_\omega$ is a $(p,s,\omega)$ atom. Moreover,
		$$
		\sum_{k=0}^\infty{|t_{k}|^{p}} < \infty \ \text{and} \ \sum_{k=0}^\infty{|s_{k}|^{p}} < \infty 
		$$
		independently of $B$. Let $B_0=E_0=B$, $B_k=B(x_B,2^kr)$, $E_k=B_k \setminus B_{k-1}$ and $M_k(x)=M(x) \chi_{{E_k}}(x)$ for $k \in \N$. 
		Let $P_k=\displaystyle{\sum_{|\alpha| \leq N_p}{m^{k}_{\alpha}\phi_{\alpha}^{k}}}$ to be the polynomial of degree at most $N_p$, restricted to the set $E_k$, for which, for every $|\alpha| \leq N_p$,
		\begin{equation} \label{AF-1}
		 \fint_{E_k}{P_k(x) \  (x-x_B)^{\alpha}dx} = m^{k}_{\alpha}\fint_{E_k}{\phi_{\alpha}^{k}(x) \  (x-x_B)^{\alpha}dx} =	m^{k}_{\alpha}:= \fint_{E_k}{M(x)(x-x_B)^{\alpha}dx}
		 \end{equation}
		and
		 \begin{equation} \label{AF-2}
		  |P_k(x)| \leq C_{n,p}\fint_{E_k}|M(x)|dx
		   \end{equation}
		for some constant independent of $k$.   This is done by choosing the polynomials $\phi_{\alpha}^{k}$ to have $\beta$-th moment equal to $|E_k|$ when $\beta = \alpha$, and zero otherwise, and noting that $(2^kr)^{|\alpha|} |\phi_{\alpha}^k(x)|\leq C$ uniformly in $k$.
		
		Setting, for $|\alpha| \leq N_p$,
		$$
		\nu_\alpha = \sum_{j = 0}^{\infty} |E_j| m^{j}_{\alpha} = \int_{\Rn} M(x)(x-x_B)^{\alpha}dx, \quad
		N_{\alpha}^{k}=\sum_{j = k+1}^{\infty} |E_j| m^{j}_{\alpha} = \int_{E_k^c} M(x)(x-x_B)^{\alpha}dx, \; \; k= 0, 1, 2,\ldots,
		$$ 
		we can represent the sum of $P_k$ as
		\begin{align*}
			\sum_{k=0}^{\infty}{P_k(x)} &= \sum_{|\alpha|\leq N_p} \left(\sum_{k=1}^{\infty} (N_{\alpha}^{k-1} - N_{\alpha}^{k}) |E_k|^{-1} \phi_{\alpha}^{k}(x) + (\nu_{\alpha} - N_{\alpha}^{0})|E_0|^{-1} \phi_{\alpha}^{0}(x)\right)\\
			 &= \sum_{|\alpha|\leq N_p}\left(\sum_{k=0}^{\infty} N_{\alpha}^{k}|E_{k+1}|^{-1} \phi_{\alpha}^{k+1}(x) - \sum_{k=0}^{\infty}N_{\alpha}^{k} |E_k|^{-1} \phi_{\alpha}^{k}(x)+ \nu_{\alpha}|E_0|^{-1} \phi_{\alpha}^{0}(x)\right)\\
			&= \sum_{k=0}^{\infty}{\sum_{|\alpha|\leq N_p}{\Phi_{\alpha}^{k}(x)}} + \sum_{|\alpha| \leq N_p}\nu_{\alpha}\, |E_0|^{-1} \, \phi_{\alpha}^{0}(x),
		\end{align*}
		where 		$$
	 \Phi_{\alpha}^{k}(x)=N_{\alpha}^{k+1} \left[ |E_{k+1}|^{-1} \phi_{\alpha}^{k+1}(x)-|E_k|^{-1}\phi_{\alpha}^{k}(x) \right],  \; \; k= 0, 1, 2,\ldots.
		$$
		Note that last sum appears since we do not have the vanishing moment conditions on the molecule. 
		
		This allows us to decompose $M$ as follows:
		\begin{align} 
			M &= \sum_{k=0}^{\infty}{(M_k-P_k)}+ \sum_{k=0}^{\infty}{\sum_{|\alpha|\leq N_p}{\Phi_{\alpha}^{k}(x)}} + \sum_{|\alpha| \leq N_p}\nu_{\alpha} \, |E_0|^{-1} \, \phi_{\alpha}^{0}(x) = S_1+S_2+S_3 . \label{M-decomposition}
		\end{align}
		We deal first with the terms $S_1$. From conditions (M1) and (M2) one gets
		\begin{equation} \label{M-j-estimation}
			\| M_k \|_{L^s} \leq 2^{\frac{\lambda}{s}} \, C_{n,s} \, |B_k|^{\frac{1}{s}-\frac{1}{p}} \,\, (2^k)^{-\frac{\lambda}{s}+n \left( \frac{1}{p}-\frac{1}{s} \right)},
		\end{equation}
		and it follows from \eqref{AF-2} that $\|P_k\|_{L^s} \leq  C_{n,p} \, \| M_k \|_{L^s}$. Moreover, from  \eqref{AF-1} we get that $M_k-P_k$ has vanishing moments up the order $N_p$.
				Thus $M_k-P_k$ is a multiple of a $(p,s)$ atom.  Writing $(M_k-P_k)(x) = t_{k} \, a_k(x)$ where $t_{k} = \| M_k-P_k \|_{L^s} \, |B_k|^{\frac{1}{p}-\frac{1}{s}}$, $a_k(x)=\dfrac{M_k(x)-P_k(x)}{\| M_k-P_k \|_{L^s}} \, |B_k|^{\frac{1}{s}-\frac{1}{p}}$ and note that from \eqref{M-j-estimation} one gets
		\begin{align} \label{control-t_j}
			\sum_{k=0}^{\infty}{|t_{k}|^p} \leq 2^{\frac{\lambda p}{s}} \, C_{n,p,s} \, \sum_{k=0}^{\infty}{(2^k)^{-\frac{\lambda p}{s}+n \left( 1-\frac{p}{s} \right)}} = C_{n,p, s,\lambda} < \infty
		\end{align}
		provided $\lambda> n \left( s/p-1 \right)$. We point out that the closer $\lambda$ gets to $n(s/p-1)$, the bigger the constant appearing in \eqref{control-t_j}. 
		
		For $S_2$, we claim that $\Phi_{\gamma}^j(x)$ is a multiple of a $(p,\infty)$ atom with full cancellation conditions.  The cancellation follows from the moment conditions on $\phi_{\alpha}^{k}$.  For the size condition, from H\"older inequality and \eqref{M-j-estimation}, for every $|\alpha|\leq N_p$ and $k\in \N$, 
		one has
		\begin{align*} 
			|N_{\alpha}^{k}|  \leq C_{n,p, s,\lambda} \, |B_{k}|^{1-\frac{1}{p}} (2^{k}r)^{|\alpha|} (2^{k})^{-\frac{\lambda}{s}+n \left( \frac{1}{p}-\frac{1}{s} \right)}.
		\end{align*}
		Hence, since $(2^kr)^{|\alpha|} |\phi_{\alpha}^k(x)|\leq C$ uniformly, it follows
		$$
		|N^{k}_{\alpha}|E_k|^{-1}\phi_{\alpha}^{k}(x)| \leq C_{n,p, s,\lambda} \,  |B_k|^{-\frac{1}{p}}(2^k)^{-\frac{\lambda}{s}+n \left( \frac{1}{p}-\frac{1}{s} \right)}.
		$$
		Therefore, writing $\displaystyle{\sum_{|\alpha| \leq N_p} \Phi_{\alpha}^{j}(x)} =s_{k} \,\, b_{k}(x)$, where $s_{k}= C_{n,p, s,\lambda}  \, (2^k)^{-\frac{\lambda}{n}+n \left( \frac{1}{p}-\frac{1}{s} \right)}$ 
		for some appropriate constants $C_{n,p, s,\lambda}$, we get that $b_k$ are $(p,\infty)$ atoms and		
		\begin{equation*}
			\sum_{k=0}^{\infty}{|s_{k}|^p} = C_{n,p, s,\lambda}\, \sum_{k=0}^{\infty}{(2^k)^{-\frac{\lambda p}{s}+n \left( 1-\frac{p}{s} \right)}} < \infty,
		\end{equation*}
		where again we used that $\lambda> n \left( s/p-1 \right)$. 
		
		Finally, for $S_3$ let
		$$
		a_\omega = \sum_{|\alpha| \leq N_p}\nu_{\alpha} \, |E_0|^{-1} \, \phi_{\alpha}^{0}(x).
		$$
		This function is supported on $E_0 = B_0$ and,  proceeding as in Remark~\ref{remark-molecules}(ii), (M1) and (M2) give
		\begin{align*}
			|\nu_{\alpha}| & = \left|\int_{\R^n}{M(x)(x-x_B)^{\alpha}dx}\right| \lesssim r^{|\alpha|+n\left( 1-\frac{1}{p} \right)}.
		\end{align*}
		The $L^s$-estimate then follows immediately from the fact that $r^{|\alpha|} |\phi_{\alpha}^0(x)|\leq C$:
		\begin{align}
			\left\| \sum_{|\alpha| \leq N_p}\nu_{\alpha} \ |E_0|^{-1}\phi_{\alpha}^{0} \right\|_{L^s} & \leq \sum_{|\alpha|\leq N_p}{|N_{\alpha}^{0}| \ |E_0|^{-1} \left( \int_{E_0}{|\phi_{\alpha}^{0}(x)|^sdx} \right)^{\frac{1}{s}}} \nonumber \\
			& \leq \sum_{|\alpha|\leq N_p}{|\nu_{\alpha}| \, |E_0|^{\frac{1}{s}-1} \ r^{-|\alpha|}} \lesssim \, r^{\,n\left(\frac{1}{s}-\frac{1}{p}\right)}. \nonumber 
		\end{align}
		It remains to show the moment conditions on $a_\omega$, which follow immediately from $(M_3)$, since by the choice of $\nu_\alpha$ and $\phi_{\alpha}^{0}$, the moments of $a_\omega$ are the same as those of $M$. Indeed, for $|\beta| \leq N_p$,
		\begin{align*}
			\int{ a_\omega(x) \, (x-x_B)^{\beta}dx} & =  \sum_{|\alpha|\leq N_p}{\nu_{\alpha}\ \left(|E_0|^{-1} \int_{E_0}{\phi_{\alpha}^{0}(x) (x-x_B)^{\beta}dx} \right)} 
			 =  \nu_{\alpha} =\int_{\R^n}{M(x)(x-x_B)^{\beta}dx}.
		\end{align*}
		Thus $a_\omega$ is a multiple of a $(p,s,\omega)$ atom.
	\end{proof}
	
Since $(p, s)$ atoms are automatically $(p,s,\lambda, \omega)$ molecules for any choice of $\lambda$ and $\omega$ in Definition~\ref{hpmolecule}, we can combine the atomic decomposition ~\ref{goldberg-decomp} with Proposition~\ref{molecular-decomp} (see also the remarks preceding Proposition~\ref{prop-psomega-atom})	 to get that $f \in h^p(\Rn)$ if and only if there exists a sequence $\{M_j\}$ of $(p,s,\lambda,\omega)$ molecules and a sequence $\{\lambda_j\} \in \ell^p(\C)$ such that 	
$\displaystyle{f = \sum_{j \in \N} \lambda_j \, M_j }$ in the sense of distributions and in the $h^p(\Rn)$ norm, and $\| f \|_{h^p}$ is comparable to the infimum of $\{ \lambda_j\}_{\ell^p}$ over all such  representations. As pointed out in the proof of the proposition, the constant in this comparison will blow up as $\lambda$ gets closer to $n(s/p-1)$.

	\section{Hardy's inequality} \label{section-hardy-inequality}
	
	In this section, as an application of the new molecules presented in Definition \ref{hpmolecule}, we prove a version of Hardy's inequality on $h^{p}(\R^{n})$ for $0<p \leq 1$, originally stated by {Goldberg in \cite{Goldberg2}.}
	
	\begin{theorem}\label{teorema4.1}
	Let $0<p \leq 1$. Then, there exists $C>0$ such that
		\begin{equation}
			\int_{\Rn} \frac{|\widehat{f}(\xi)|^p}{(1+|\xi|)^{n(2-p)}}d\xi \leq C \| f \|_{h^p}^{p} \quad \forall \ f \in h^p(\Rn).
		\end{equation}
	\end{theorem}
	
	The case $n=1$ and $p=1$ was proved by Dafni and Liflyand in \cite{GaliaLiflyand}. The main ingredient is the following pointwise  control of the Fourier transform of a  $(p,s,\lambda,\omega)$ molecule on $h^p(\Rn)$.  This result resembles the decay of the  Fourier transform decay for standard atoms on $H^{p}(\Rn)$ (see \cite[Theorem 7.20 p. 337]{GarciaFranciaWeighted}), but we need to account for the non-vanishing moments.  Moreover, due to the weaker decay at infinity of the molecules, namely condition (M2) compared to compact support, we cannot get unlimited smoothness of the Fourier transform.  The parameter $\lambda$ is what determines this limitation.

	\begin{lemma} \label{control-fourier-molecule}
		Let $0<p\leq 1 \leq s< \infty$ with $p<s$ and $\lambda> n \left(\frac{s}{p}-1  \right)$.  Suppose $M$ satisfies conditions (M1) and (M2) with respect to the ball $B=B(x_B,r) \subset \Rn$.  Then the Fourier transform of $M$ satisfies
		\begin{align}
		\label{eq-FT}
			|\widehat{M}(\xi)| &\lesssim \, |\xi|^{\gamma} \, r^{\gamma - \gamma_p} + \sum_{|\alpha| \leq N} \, |\xi|^{|\alpha|} \left| \int_{\Rn} M(x)(x-x_B)^{\alpha}dx \right|
		\end{align}
		for any $\gamma \in (\gamma_p,\frac \lambda s - \frac n{s'})$ and integer $N$ with $N  < \gamma \le  N + 1$.
	\end{lemma}
	
	\begin{proof}
	
		Since the absolute value of the Fourier transform is preserved under translation of the function, we may assume $x_B = 0$. For $\xi = 0$, we see that equality holds by in \eqref{eq-FT} by considering the  $\alpha = 0$ term in the sum on the right-hand-side, so we need only  consider $\xi \neq 0$.

	Suppose first that $\gamma = N + 1 < \frac \lambda s - \frac n{s'}$.  Denoting $e^{-2\pi i x \cdot \xi}$ by $\varphi(x)$ , we write 
	$P_{N,\varphi, 0}(x) = \sum_{|\alpha| \leq N} C_{\alpha} \, (\partial^{\alpha}\varphi)(0)x^{\alpha}$  for its Taylor polynomial of order $N$ at $0$, and use the formula for the remainder to get, for $t \in (0,1)$,
		\begin{align}
	|\widehat{M}(\xi)| 
	\nonumber	& =  \left| \int_\Rn M(x)  \left[\varphi(x) - P_{N,\varphi, 0}(x)  \right] dx  +  \sum_{|\alpha| \leq N} C_{\alpha} \, (\partial^{\alpha}\varphi)(0) \int_\Rn M(x)x^{\alpha}dx  \right| \\
	\nonumber	& \le	\left| \int_\Rn M(x) \sum_{|\alpha| = N+1} C_{\alpha} \,  (\partial^{\alpha}\varphi)(tx)x^{\alpha}dx\right|  +  \sum_{|\alpha| \leq N} C_{\alpha} \, |2\pi\xi|^{|\alpha|} \left| \int_\Rn M(x)x^{\alpha}dx  \right| \\
	\label{equation-1}		&\lesssim  |\xi|^{N + 1}\int_\Rn |M(x)|\, |x|^{N + 1}dx +  \sum_{|\alpha| \leq N}  |\xi|^{|\alpha|} \left| \int_\Rn M(x)x^{\alpha}dx  \right|. 
		\end{align}
		Similarly to Remark~\ref{remark-molecules} (ii), from conditions (M1) and (M2) of the molecule and H\"older's inequality, one has
		\begin{align*}
			\int_{\Rn} |M(x)| \, |x|^{N+1}dx 
			&\leq r^{\frac{n}{s'}+N+1} \, \| M \|_{L^s(B)} + \| M \, |\cdot|^{\frac{\lambda}{s}} \, \|_{L^s(B^c)} \, \| \, |\cdot|^{-\frac{\lambda}{s}+N +1} \|_{L^{s'}(B^c)} \\
			& = r^{\frac{n}{s'}+N+1} \, \| M \|_{L^s(B)} + r^{-\frac{\lambda}{s}+N+1+\frac{n}{s'}} \, \| M \, |\cdot|^{\frac{\lambda}{s}} \, \|_{L^s(B^c)} \\
			& \lesssim r^{N+1-\gamma_p},
		\end{align*}
		where the convergence of the integral follows from the assumption that  $N + 1 < \frac \lambda s - \frac n{s'}$. This gives the result in the case $\gamma =  N + 1$.
		
		Now suppose $\gamma < N + 1$.  Recalling that $\xi \neq 0$, we write
	 	$$
	\widehat{M}(\xi)  = \int_{|x| \ge |\xi|^{-1}}{e^{-2\pi i x \cdot \xi}}M(x)dx+\int_{|x| \le |\xi|^{-1}}{e^{-2\pi i x \cdot \xi}}M(x)dx =: I_1 + I_2. 
$$
We estimate the first integral using H\"older's inequality, together with the bound the $L^s(\Rn)$ norm of $M(x)|x|^{\frac{\lambda'}{s}}$ with $\lambda' = s(\gamma +  \frac n{s'}) < \lambda$ (see Remark~\ref{remark-molecules} (i)), as follows:
	\begin{align*} |I_1| &\leq \int_{|x| \ge |\xi|^{-1}}|M(x)| dx \le
	  \| M \, |\cdot|^{\frac{\lambda'}{s}} \, \|_{L^s(\Rn)} \, \|\,|\cdot|^{-\frac{\lambda'}{s}}\|_{L^{s'}({|x| \geq |\xi|^{-1})}} \le r^{\gamma-\gamma_p} |\xi|^{\gamma}.
	\end{align*}
	For the second integral, we again proceed via the Taylor expansion of $\varphi(x)=e^{-2\pi i x \cdot \xi}$, to get, as in \eqref{equation-1}
	\begin{align*} |I_2| &\lesssim  |\xi|^{N + 1}\int_{|x| \le |\xi|^{-1}} |M(x)|\, |x|^{N + 1}dx +  \sum_{|\alpha| \leq N}  |\xi|^{|\alpha|} \left|\int_{|x| \le |\xi|^{-1}} M(x)x^{\alpha}dx  \right| \\
	&\le   |\xi|^{N + 1}\int_{|x| \le |\xi|^{-1}} |M(x)|\, |x|^{\frac{\lambda'}s}|x|^{N + 1 -\frac{\lambda'}s}dx + \sum_{|\alpha| \leq N}  |\xi|^{|\alpha|}  \left| \int_\Rn M(x)x^{\alpha}dx - \int_{|x| \ge |\xi|^{-1}} M(x)x^{\alpha}dx  \right| \\
	&\le  |\xi|^{N + 1}  \| M \, |\cdot|^{\frac{\lambda'}{s}} \, \|_{L^s(\Rn)} \, \| \, |\cdot|^{N+1-\frac{\lambda'}{s}} \|_{L^{s'}(|x|\leq |\xi|^{-1})} \\
	& \quad \quad \quad + \sum_{|\alpha| \leq N}  |\xi|^{|\alpha|}  \int_{|x| \ge |\xi|^{-1}} |M(x)||x|^{\alpha}dx  
	+\sum_{|\alpha| \leq N}  |\xi|^{|\alpha|}  \left| \int_\Rn M(x)x^{\alpha}dx \right| \\
& \lesssim  |\xi|^{N + 1} r^{\frac{\lambda'}{s}-\frac{n}{s'}-\gamma_p}  |\xi|^{-(N + 1 -\frac{\lambda'}s +  \frac{n}{s'})} 
+ \sum_{|\alpha| \leq N}  |\xi|^{|\alpha|}   \| M \, |\cdot|^{\frac{\lambda'}{s}} \, \|_{L^s(\Rn)} \, \| \, |\cdot|^{|\alpha| - \frac{\lambda'}{s}} \|_{L^{s'}(|x| \ge |\xi|^{-1})} \\
& \quad \quad \quad +\sum_{|\alpha| \leq N}  |\xi|^{|\alpha|}  \left| \int_\Rn M(x)x^{\alpha}dx \right| \\
&\lesssim r^{\gamma - \gamma_p} |\xi|^{\gamma}
 +\sum_{|\alpha| \leq N}  |\xi|^{|\alpha|}  \left| \int_\Rn M(x)x^{\alpha}dx \right|. \\
	\end{align*}
	Here we have used that $\gamma = \frac {\lambda'} s - \frac n{s'} < N + 1$ for the local integrability and that $|\alpha| \leq N < \gamma = \frac{\lambda'}{s}-\frac{n}{s'}$ implies $s'(|\alpha| - \frac{\lambda'}{s}) < -n$.
	This concludes the case  $\gamma  < N + 1$.
	\end{proof}
	
For a molecule, the above estimate on the Fourier transform and the control of the moments allows us to prove the following more refined version of Hardy's inequality:

\begin{lemma}
\label{eq-Hardy}
If $1\leq s\leq 2$ with $p<s$ and $M$ is a $(p,s,\lambda, \omega)$ molecule in $h^p(\Rn)$ associated to the ball $B=B(x_B,r)$, then for $a > 0$, 
		\begin{equation}
		\label{eq-kappa}
			\int_{\Rn} \frac{|\widehat{M}(\xi)|^p}{(a\omega+|\xi|)^{n(2-p)}} d\xi \leq C_{a, \omega}.
		\end{equation}
\end{lemma}
	 
In the homogeneous case, $\omega = 0$, we recover Hardy's inequality for $H^p(\R^n)$. For $\omega > 0$, taking $a = \omega^{-1}$ shows that Goldberg's Hardy inequality holds uniformly for molecules with a constant depending on $\omega$. Applying the Fourier transform to the molecular decomposition of $f$ gives the result of the Theorem~\ref{teorema4.1}. 	

		\noindent \textit{Proof of Lemma~\ref{eq-Hardy}. }
		To show \eqref{eq-kappa} we split integral in the following way:
		$$
		\int_{\Rn} \frac{|\widehat{M}(\xi)|^p}{(a \omega+|\xi|)^{n(2-p)}} d\xi = \int_{|\xi|<r^{-1}} + \int_{|\xi|>r^{-1}} := I_1+I_2.
		$$
		\textbf{Control of $I_2$:} 
		Applying the H\"older and Hausdorff-Young inequalities, one gets
		\begin{align*}
			\int_{|\xi|>r^{-1}} \frac{|\widehat{M}(\xi)|^p}{(a \omega +|\xi|)^{n(2-p)}} d\xi 
			& \leq \|\widehat{{M}}\|_{s'}^{p} \, \left( \int_{|\xi|>r^{-1}}{|\xi|^{-\frac{n(2-p)}{1-p/s'}}d\xi} \right)^{1-\frac{p}{s'}}  \\
			& \lesssim \|M\|_{s}^{p}  \, r^{n(2-p) - n(1-\frac{p}{s'})}\left( \int_{|\xi|>1}{|\xi|^{-\frac{n(2-p)}{1-p/s'}}d\xi} \right)^{1-\frac{p}{s'}}  \\
			&\leq C.
		\end{align*}
		Here we've used condition (M1), and the integrability of the second term follows since
		$$
		1>p\left(1-\frac{1}{s'}\right) \quad \Leftrightarrow \quad -\frac{n(2-p)}{1-p/s'} < -n.
		$$
		\textbf{Control of $I_1$:} Taking $N = N_p$ and $\gamma \in  \left(\gamma_p,\frac \lambda s - \frac n{s'}\right) \cap \left(N_p, N_p + 1\left]\right.\right.$  in Lemma~\ref{control-fourier-molecule}, one has
			\begin{align*}
			I_1 &\lesssim r^{p(\gamma - \gamma_p)} \int_{|\xi|<r^{-1}}|\xi|^{p\gamma} (1+|\xi|)^{n(p-2)}d\xi \\
			& \quad \quad \quad \quad \quad \quad \quad \quad \quad \quad +  \sum_{|\alpha| \leq N_p} \left| \int_{\Rn}M(x)(x-x_B)^{\alpha}dx \right|^p \int_{|\xi|<r^{-1}} |\xi|^{|\alpha|p} (1+|\xi|)^{n(p-2)} d\xi. \\
			&\quad \quad \quad \quad \quad \quad \quad \quad \quad \quad := I_3 + I_4
		\end{align*}
		For $I_3$, using that $(1+|\xi|)^{n(p-2)} \leq |\xi|^{n(p-2)}$ we get
		\begin{align*}
			 |I_3| & \le r^{p(\gamma - \gamma_p)}  \int_{|\xi|<r^{-1}} |\xi|^{n(p-2)+p\gamma}d\xi \\
			& \simeq r^{\,p\gamma +  n(p-1)} \, r^{-p\gamma-n(p-2)-n} = 1,
		\end{align*}
		where the integrability follows from $p\gamma > p\gamma_p=n(1-p)$.
		
		For $I_4$,  using the approximate moment conditions (M3) of the molecule when $\omega > 0$, we get
		\begin{align*}
			\sum_{|\alpha|\leq N_p} & \left| \int_{\Rn} M(x)(x-x_B)^{\alpha}dx\right|^{p}  \int_{|\xi|<r^{-1}} |\xi|^{|\alpha|p} (a \omega+|\xi|)^{n(p-2)} d\xi  \\
			& = \sum_{|\alpha|\leq N_p} \left| \int_{\Rn} M(x)(x-x_B)^{\alpha}dx\right|^{p} 
			 a \omega^{np - n + |\alpha|p}\int_{|\xi|<(a \omega r)^{-1}} |\xi|^{|\alpha|p} (1+|\xi|)^{n(p-2)} d\xi 
			 \\
			&\leq  \sum_{|\alpha|\leq N_p} 
			\left| \int_{\Rn} M(x)(x-x_B)^{\alpha}dx\right|^{p}
			 (a \omega)^{np - n + |\alpha|p} \int_{1}^{1+(a \omega r)^{-1}} t^{p|\alpha|+np-n-1}dt \\
			& \leq \sum_{|\alpha|<\gamma_p} \omega^p \,  (a\omega)^{(|\alpha|-\gamma_p)p}\int_{1}^{\infty} t^{(|\alpha| - \gamma_p)p-1}dt + \sum_{\substack{|\alpha|=\gamma_p = N_p}} \left[\log \left(1+\frac{1}{\omega r} \right) \right]^{-1} 
			\int_{1}^{1+(a\omega r)^{-1}} t^{-1}dt \\  
			&\leq C_{a,\omega,p}  + \sum_{\substack{|\alpha|=N_p \\ N_p \in \Z}}\left[ \log \left(1+\frac{1}{\omega r} \right) \right]^{-1} \log \left(1+ \frac{1}{a \omega r} \right) \\
			& \leq C_{a,\omega}.
		\end{align*}

	\section{Continuity of inhomogeneous versions of Calder\'on-zygmund-type operators} \label{section-applications}
	
	In this section, we present new results on the continuity of inhomogeneous Calder\'on-Zygmund-type operators on $h^p(\Rn)$ as an application of the approximate atom and molecule tools presented in Section \ref{section-approximate-atoms-molecules}.
 
	\label{applications}
	\subsection{Boundedness of inhomogeneous Calder\'on-Zygmund-type operators} \label{subsection-ICZO}
	
	Ding, Han and Zhu in \cite{DingHanZhu2020} introduced the following inhomogeneous version of classical Calder\'on-Zygmund operators.  A locally integrable function defined on $\R^{2n}$ away from the diagonal is called an {\em $(\mu, \delta)$-inhomogeneous standard kernel} if there exist $\mu>0$ and $0<\delta\leq 1$ such that
	\begin{equation} \label{inhomogeneous-kernel}
		|K(x,y)| \leq C \ \min{\left\{ \frac{1}{|x-y|^n}, \ \frac{1}{|x-y|^{n+\mu}} \right\}}, \quad \quad x \neq y,
	\end{equation} 
	and
	\begin{equation} \label{holder-condition}
		|K(x,y)-K(x,z)|+|K(y,x)-K(z,x)| \leq C \dfrac{|y-z|^{\delta}}{|x-z|^{n+ \delta}}
	\end{equation}
	for all $|x-z| \geq 2|y-z|$. A linear and bounded operator $T: \mathcal{S}(\R^{n}) \rightarrow \mathcal{S}'(\R^{n})$ is an \textit{inhomogeneous Calder\'{o}n-Zygmund operator} if the following properties are satisfied:
	\begin{enumerate}
		\item[(i)] $T$ extends to a continuous operator from $L^2(\Rn)$ to itself; 
		\item[(ii)] $T$ is associated to an $(\mu, \delta)-$inhomogeneous standard kernel given (formally) by 
		$$\langle Tf,g \rangle = \int{\int{K(x,y)f(y)g(x)dy}dx}, \quad \quad \forall \ f,g \in \mathcal{S}(\Rn) \mbox{ with disjoint supports.} $$
	\end{enumerate}
	
	An important example of \textit{inhomogeneous Calder\'{o}n-Zygmund operators} are the standard pseudodifferential operators in the H\"ormander class $OpS^{0}_{1,0}(\Rn)$. In fact, it is well known (see \cite[Theorem 1.1 (e)]{AlvarezHounie}) that the kernel associated to this class of operators satisfies the pointwise estimate 
	$$
	|\partial_{y} K(x,y)| \leq C \, |x-y|^{-1-n}, \quad \quad x \neq y,
	$$ 
	and then \eqref{holder-condition} follows by the mean value theorem. The estimate \eqref{inhomogeneous-kernel} can be directly verified from the pseudo-local property \cite[Theorem 1.1 (a)]{AlvarezHounie}.
	
	We now establish results on the continuity of inhomogeneous Calder\'{o}n-Zygmund operators on  $h^p(\Rn)$ for $0<p\leq 1$ assuming a weaker $L^s$ integral-type condition on the kernel, $1 \leq s < \infty$, given by
	\begin{equation} \label{Ls-hormander}
		\displaystyle{\left( \int_{A_j(z,r)}{|K(x,y)-K(x,z)|^s+|K(y,x)-K(z,x)|^sdx} \right)^{\frac{1}{s}} \lesssim |A_j(z,r)|^{\frac{1}{s}-1} \, 2^{-j\delta}} \quad\mbox{ for } |y - z| < r,
	\end{equation}
	where $0<r<1$ , $A_j(z,r) := \left\{ x \in \Rn: \ 2^{j}r \leq |x-z| < 2^{j+1}r \right\}$,  $j \in \N$ and $\Gd>0$. Condition \eqref{Ls-hormander} is weaker than \eqref{holder-condition}; moreover, an $L^{s_{1}}$ condition  is stronger than an $L^{s_{2}}$ condition provided $s_{1}>s_{2}$, hence the $L^1$ integral-type condition is the weakest one. Continuity results for this class of operators on $H^{p}(\R^{n})$ were recently studied by {the last two authors} in \cite{VasconcelosPicon}.
	
	Next, we want to express the cancellation condition in term of $T^{\ast}(x^{\alpha})$ (for further details see \cite[p. 23]{MeyerCoifman1997}).
	
	\begin{definition} \label{T*def}
		Let $N \in \mathbb{N} \cup \{0\}$. We denote by $L^{2}_{c,N}(\Rn)$ the space of all $g \in L^2(\Rn)$ compactly supported functions such that $\int{g(x)x^{\alpha}}=0$ for all $|\alpha| \leq N$.  For such an $\alpha$, define $T^{*}(x^{\alpha})$ in the distributional sense by
		\begin{equation} \label{T*}
			\langle T^{*}(x^{\alpha}), g \rangle = \langle x^{\alpha}, T(g)  \rangle = \int_{\R^{n}}x^{\alpha}Tg(x)dx, \quad \quad \forall\, g \,\in L^{2}_{c,N}(\R^n). 
		\end{equation} 
	\end{definition}
	
	\noindent The space  $L^{2}_{c,N_p}(\Rn)$ corresponds to multiples of $(p,2)-$atoms in $H^p(\Rn)$. That $T^{*}(x^{\alpha})$ is well-defined by \eqref{T*} has been stated for standard Calder\'on-Zygmund operators in \cite[p. 23]{MeyerCoifman1997}. The next proposition extends it to the type of operators we are considering in this work. 
	
	\begin{proposition} 
	\label{estimative-L1-moments}
		Let $T$ be a linear and bounded operator on $L^2(\Rn)$ whose associated kernel satisfies \eqref{inhomogeneous-kernel} and \eqref{Ls-hormander} with $s=1$. Then $x^{\alpha}Tg(x) \in L^1(\Rn)$ for all $g \in L^{2}_{c,0}(\Rn)$ provided that $|\alpha|  < \min\{ \mu, \, \delta\}$.
	\end{proposition}
	
	\begin{proof}
		Let $g \in L^{2}_{c,0}(\Rn)$, fix a ball  $B=B(x_B,r)$ containing the support of $f$,  and write
		$$
		\int_{\Rn}{|x^{\alpha}Tf(x)|dx} = \int_{2B}{|x^{\alpha}Tg(x)|dx}+\int_{(2B)^c}{|x^{\alpha}Tg(x)|dx}.
		$$
		From the boundedness of $T$ on $L^2(\Rn)$ we get
		\begin{align*}
			\int_{2B}{|x^{\alpha}Tf(x)|dx}  \leq  \| x^{\alpha} \|_{L^{\infty}(2B)} \ |2B|^{\frac{1}{2}} \ \| Tg \|_{L^2} 
			\lesssim   r^{|\alpha|+\frac{n}{2}}  \| g \|_{L^2} <\infty.
		\end{align*}
		Suppose now that $r \geq 1$. The estimation of the second integral follows by \eqref{inhomogeneous-kernel}:
		\begin{align*}
			\int_{(2B)^c}{|x|^{|\alpha|}|Tg(x)|dx} & \leq \sum_{j \in \N}\int_{A_j(x_B,r)} \int_{B} |K(x,y)| \, |g(y)| \, (|x_B| + 2^jr)^{|\alpha|} dydx \\
			& \lesssim \| g \|_{L^2} \, r^{\frac{n}{2}} \, \sum_{|\alpha'|\leq |\alpha|} |x_B|^{|\alpha|-|\alpha'|} \,\sum_{j \in \N}  (2^jr)^{|\alpha'|}  \int_{A_j(x_B,r)} |x-x_B|^{-n-\mu}dx \\
			& \lesssim \| g \|_{L^2} \, \sum_{|\alpha'|\leq |\alpha|} |x_B|^{|\alpha|-|\alpha'|} \, r^{|\alpha'|+\frac{n}{2}-\mu} \sum_{j \in \N} (2^j)^{|\alpha'|-\mu}<\infty
		\end{align*}
		since $|\alpha'| \leq |\alpha| < \mu$. 
		
		For $r<1$, since $g$ has vanishing integral, the bound follows by applying \eqref{Ls-hormander}:
		\begin{align*}
			\int_{(2B)^c}{|Tg(x)x^{\alpha}|dx} & \leq \sum_{j \in \N}\int_{A_j(x_B,r)} \int_{B} |K(x,y)-K(x,x_B)| \, |g(y)| \, |x|^{|\alpha|} dydx \\
			& \lesssim \| g \|_{L^2} \, \sum_{|\alpha'|\leq |\alpha|} |x_B|^{|\alpha|-|\alpha'|} \, r^{|\alpha'|+\frac{n}{2}-\delta} \sum_{j \in \N} (2^j)^{|\alpha'|-\delta} < \infty
		\end{align*}
		since $|\alpha'|  < \delta$. This completes the proof. 
	\end{proof}
	
	In \cite[Theorem 1.1]{DingHanZhu2020} the authors considered the boundedness of inhomogeneous Calder\'on-Zygmund operators on $h^p(\Rn)$ for some $\frac{n}{n+1}<p<1$ assuming the cancellation condition $T^{\ast}(1) \in \dot{\Lambda}_{\gamma_p}(\R^{n})$. In the next theorem, we extend this result for full range $0<p \leq 1$, assuming an $L^s$ integral-type condition on the kernel and a local Campanato-type cancellation condition. In particular, our result recover the one in \cite{DingHanZhu2020}  when $\frac{n}{n+1}<p<1$.  Note that when $p$ is small, the values of $\mu$ and $\delta$ must be large.
	
\begin{theorem} \label{continuity-inhomCZO}
		Let $0<p\leq 1$ and $T$ be an inhomogeneous Calder\'on-Zygmund operator associated to a kernel satisfying the integral condition \eqref{Ls-hormander} for some $1\leq s\leq 2$ with $p<s$. Then $T$ can be extended to a bounded operator from $h^p(\Rn)$ to itself provided that $\min\{\mu,\delta\} > \gamma_p$ and there exists $C>0$ such that for any ball $B\subset \Rn$ and $\alpha \in \Z_+^n$ with $|\alpha| \leq N_p$,  
\begin{equation}
 \label{localCamp}
f= T^*[(\cdot-x_B)^{\alpha}] \quad \mbox{ satisfies } \quad \left(\fint_{B}|f(y)-P^{N_p}_{B}(f)(y)|^2dy\right)^{1/2}\leq C \, \Psi_{p,\alpha}(r(B)),
\end{equation}
where $P_{B}^{N_p}(f)$ is the  polynomial of degree $\leq N_p$ that has the same moments as $f$ over $B$ up to order $N_p$, and for $\varphi_p$ is as in Definition~\ref{approximate-atom},\\
$$
\displaystyle{\Psi_{p,\alpha}(t):=\left\{ \begin{array}{ll} 
t^{\gamma_p} &\quad \text{if } |\alpha|< \gamma_p, \\
 t^{\gamma_p} \varphi_p(t) &\quad \text{if } |\alpha| = \gamma_p = N_p \in \Z.
			\end{array} \right.}
$$
	\end{theorem}
	
	\begin{remark}
	\textnormal{Note that $f = T^*[(\cdot-x_B)^{\alpha}]$ is required to satisfy the Campanato-type condition \eqref{localCamp} only for the ball $B$ with respect to which it is defined, but the condition makes sense for any ball $B'$. This is because the  hypotheses on $T$ together with Proposition~\ref{estimative-L1-moments} imply that  for every ball $B'$,  \eqref{T*} defines $f$ as an element of the dual space of $L^2_{N_p}(B')$ (the  functions in $L^2(B')$ with vanishing moments up to order $N_p$), which can be identified with the quotient space of $L^2(B')$ by the polynomials of degree up to $N_p$.   Thus one can replace \eqref{localCamp}  by the stronger condition that  $f \in \dot{\Lambda}_{\gamma_{p}}(\Rn)$ if $|\alpha|<\gamma_p$ and $f \in L^{2,\Psi_p}_{N_{p}}(\Rn)$ 
	 if $\gamma_p \in \Z$ and $|\alpha|=\gamma_p$.}
	\end{remark}
	
	\begin{proof}
		Let $a$ be a $(p, 2)$ atom in the sense of Definition \ref{goldberg-atom}, supported in $B:=B(x_B,r)$. We will show that $Ta$ is a $(p, s, \lambda,\omega)$ molecule for $\lambda$ satisfying
		$\gamma_p < \frac \lambda s - \frac n{s'} < \min\{\mu,\delta\}$.  By Proposition~\ref{molecular-decomp} and the $h^p$ analogue of the results of \cite{MedaSjogrenVallarino2008}, this suffices to show the boundedness of $T$ on $h^p$.

		 As  $1\leq s \leq 2$, condition (M1) follows from $L^2$-continuity of $T$:
		\begin{equation} \label{estimate-M1-Ta}
			\int_{2B}{|Ta(x)|^sdx} \lesssim |B|^{1-\frac{s}{2}} \|Ta\|_{L^2}^{s} \lesssim |B|^{1-\frac{s}{2}} \|a\|_{L^2}^{s} \lesssim r(B)^{n(1-\frac{s}{p})}.
		\end{equation}
		For (M2), suppose first $r\geq 1$. From condition \eqref{inhomogeneous-kernel} it follows that for $|x-x_B|>2r$ we have $ |K(x,y)|  \lesssim |x-x_B|^{-n-\mu}$ for all $y \in B$ and therefore
		\begin{align*}
			|Ta(x)| \leq \int_{B} |K(x,y)| \, |a(y)| dy \lesssim \| a \|_{L^2} |B|^{1/2} \int_{B} |x-x_B|^{-n-\mu} \lesssim r^{-\gamma_p} |x-x_B|^{-n-\mu}. 
		\end{align*}
		Then, for $\lambda$ satisfying $\frac \lambda s - \frac n{s'} < \mu$, which means $\lambda - s(n+ \mu) < -n$, we have
		\begin{align*}
			\int_{(2B)^c}{|Ta(x)|^s \, |x-x_B|^{\lambda}dx} \, &\lesssim r^{-s\gamma_p} \int_{(2B)^c} |x-x_B|^{\lambda-s(n+\mu)} dx 
			\lesssim r^{\, \lambda+ n \left(1-\frac{s}{p}\right)} \, r^{-s\mu} 
			\lesssim r^{\, \lambda+ n \left(1-\frac{s}{p}\right)}.
		\end{align*}
		 Condition (M3) follows from Remark \ref{remark-molecules} item (ii).
		
		 Suppose now $r<1$. Using the vanishing moment condition of the atom and  \eqref{Ls-hormander} we can apply Minkowski inequality for integrals to get
		\begin{align*}
			\int_{(2B)^c}&|Ta(x)|^{s}|x-x_B|^{\lambda}dx = \sum_{j=0}^{\infty}{\int_{A_j(x_B,r)}{\left| \int_{B}{[K(x,y)-K(x,x_B)]a(y)dy} \right|^{s} \, |x-x_B|^{\lambda}dx}} \\
			& = \sum_{j=0}^{\infty}{\left\{ \left[ \int_{A_j(x_B,r)}{ \left( \int_{B}{|K(x,y)-K(x,x_B)|\,\, |a(y)| \,\, |x-x_B|^{\frac{\lambda}{s}}dy} \right)^{s}  dx}   \right]^{\frac{1}{s}} \right\}^{s}} \nonumber \\
			& \leq  \sum_{j=0}^{\infty} (2^{j+1}r)^{\lambda}{ \left\{ \int_{B} |a(y)|{ \left[ \int_{A_j(x_B,r)}{|K(x,y)-K(x,x_B)|^{s}\,\,dx} \right]^{\frac{1}{s}}dy } \right\}^{s}} \nonumber \\
			& \lesssim  \sum_{j=0}^{\infty}{(2^jr)^{\lambda} \,\, (2^jr)^{-n(s-1)} \,\, 2^{-j s\delta}} \,\,  \left( \int_{B}{|a(y)|dy} \right)^{s} \nonumber \\
			& \leq  \sum_{j=0}^{\infty}{(2^jr)^{\lambda} \,\, (2^jr)^{-n(s-1)} \,\, 2^{-j s\delta}} \,\, r^{-s\gamma_p} \nonumber \\
			& = C \ r^{\lambda + n\left( 1-\frac{s}{p} \right)} \sum_{j=0}^{\infty}{2^{j\left[\lambda-n(s-1)-s\delta \right]}} \nonumber \\
			& =  C \ r^{\lambda + n\left( 1-\frac{s}{p} \right)} \nonumber 
		\end{align*}
		assuming $\lambda < n(s-1)+s\delta$, which is the same as $\frac \lambda s - \frac n{s'} < \delta$.
		Finally, in order to verify that (M3) holds, note that for $r<1$ the function $a$ is in particular a $(p,2)$ atom in $H^p(\Rn)$ with full cancellation condition. 
		From condition \eqref{localCamp}, setting $f= T^*[(\cdot-x_B)^{\alpha}]$, we have, by \eqref{T*},
		\begin{align*}
			\left| \int{Ta(x)(x-x_B)^{\alpha}dx} \right| &= \left| \langle f, a \rangle \right|  \leq  \left(\int_{B}|f(y)-P^{N_p}_{B}(f)(y)|^2dy\right)^{1/2} \|a\|_{L^2(B)}\\
			&\lesssim \Psi_p(r)|B|^{1/2} \|a\|_{L^2(B)}\\
			&\lesssim \Psi_p(r)r^{-\gamma_p}
		\end{align*}
	which is bounded by $C_{n,p}$ if $|\alpha|< \gamma_p$ and by $\varphi_p(r(B))$ if  $|\alpha|=\gamma_p$ from  $\eqref{localCamp}$.
	\end{proof}

	As corollary, we recover part of \cite[Theorem 1.1]{DingHanZhu2020}.
	
	\begin{corollary}
		Let T be an inhomogeneous Calder\'on-Zygmund operator satisfying the pointwise controls  \eqref{inhomogeneous-kernel} and \eqref{holder-condition}. If $T^{\ast}(1) \in \dot{\Lambda}_{\gamma_p}(\Rn)$, then $T$ is a bounded operator from $h^p(\Rn)$ to itself for $n/(n+\min\{\delta, \, \mu \}) < p < 1$.
	\end{corollary}

	\subsection{Strongly singular inhomogeneous Calder\'on-Zygmund operators} \label{subsection-ISSCZO}
	
	The study of strongly singular Calder\'on-Zygmund operators was first motivated by some classes of multiplier operators whose symbol is given by $e^{i|\xi|^{\sigma}}/|\xi|^{\beta}$ away from the origin, for some parameters $\sigma$ and $\beta$. It is well know that these operators satisfy some $L^p$-estimates for a restricted range of $p$ (see \cite{Fefferman1970,Hirschman1959,Wainger1965}). To study some endpoint $L^p$-estimates, C. Fefferman in \cite{Fefferman1970} enlarged this class of multipliers into a class of convolution operators named weakly strongly singular integrals. 
	 In  \cite{AlvarezMilman}, \'Alvarez and Milman defined the non-convolution version of these operators and showed it has connection to more general classes of pseudodifferential operators in the H\"ormander class. In this section we present an inhomogeneous version of the strongly singular Calder\'on-Zygmund operators and we study a continuity result on $h^{p}(\R^{n})$.
	
	We say that a linear and bounded operator $T: \mathcal{S}(\Rn) \rightarrow\mathcal{S}'(\Rn)$ is an \textit{inhomogeneous strongly singular Calder\'on-Zygmund operator} if its distributional kernel is given by a continuous function $K(x,y)$ on $\R^{2n}$ away the diagonal and   satisfies, for some $\mu>0$
	\begin{equation} \label{size-kernel-strong}
		|K(x,y)| \leq C \ \min{\left\{ \frac{1}{|x-y|^n}, \ \frac{1}{|x-y|^{n+\mu}} \right\}} \quad \text{for}\,\, x \neq y,
	\end{equation}
	and for some $0<\delta\leq 1$ and $0<\sigma\leq 1$,
	\begin{equation} \label{holder-kernel-strong}
		|K(x,y)-K(x,z)|+|K(y,x)-K(z,x)| \leq C \, \frac{|y-z|^{\delta}}{|x-z|^{n+\frac{\delta}{\sigma}}}
	\end{equation}
	if $|x-z|>2|y-z|^{\sigma}$. In addition, we also assume that $T$ has the following boundedness properties:
	\begin{enumerate}
		\item[(i)] $T$ can be extended to a bounded operator from $L^2(\Rn)$ to itself;
		\item[(ii)] For some $\beta \in  \left[\left.(1-\sigma)\frac{n}{2}, \frac{n}{2}\right)\right.$, $T$ can be extended to a bounded operator from $L^q(\Rn)$ to $L^2(\Rn)$, where $\displaystyle{\frac{1}{q} = \frac{1}{2}+\frac{\beta}{n}}$.
	\end{enumerate}
	
	\begin{example}
		\textnormal{It follows from \cite[Theorem 1.1 (a) and (e)]{AlvarezHounie} that $T \in OpS^{-n(1-\sigma)}_{\sigma, b}(\Rn)$ for $0<b \leq \sigma <1$ satisfies the pointwise condition \eqref{holder-kernel-strong} for $\delta=1$ and \eqref{size-kernel-strong} for any $\mu>0$ sufficiently large. In addition, $T$ satisfies the boundedness conditions (i) and (ii) above (see \cite[Theorem 3.5]{AlvarezHounie}).}
	\end{example}
	
	In the same spirit of the $L^s$ integral-type condition \eqref{Ls-hormander}, we assume the weaker kernel condition
	\begin{align} \label{1-hormander-2a}
		&\left( \int_{A_j(z,r^{\rho})}{|K(x,y)-K(x,z)|^s+|K(y,x)-K(z,x)|^sdx} \right)^{\frac{1}{s}} \lesssim |A_j(z,r^{\rho})|^{\frac{1}{s}-1+\frac{\delta}{n} \left(\frac{1}{\rho}-\frac{1}{\sigma}  \right)}  2^{-\frac{j\delta}{ \rho}}
	\end{align}
	for $0<r<1$, $\delta>0$ where  $A_j(z,\tilde{r})=\{ x \in \Rn: \ 2^j \tilde{r} \leq |x-z| < 2^{j+1}\tilde{r} \}$, $0< \rho \leq \sigma \leq 1$, $z \in \Rn$ and $|y-z|<r$. If $\rho \leq \sigma$ then \eqref{holder-kernel-strong} implies \eqref{1-hormander-2a}. Conditions of this type were considered in \cite[inequalities (3.7) and (3.8) p. 412]{AlvarezMilmanVectorValued} and {\cite{VasconcelosPicon}} and are  naturally related to pseudodifferential operators, as pointed out in the next example.
	
	\begin{example}
		\textnormal{From \cite[Theorem 5.1]{AlvarezHounie} it follows that the associated kernel of operator in the H\"ormander class $OpS^{m}_{\sigma, b}(\Rn)$ for $0<b \leq \sigma <1$ satisfies the estimate \eqref{1-hormander-2a} with $s=1$ for $m \leq -\frac{n}{2}(1-\sigma)$. The generalization for $1<s\leq 2$ follows from \cite{VasconcelosPicon}.}
	\end{example}

	\begin{theorem}
		Let $0<p\leq 1$ and $T$ an inhomogeneous strongly singular Calder\'on-Zygmund operator whose kernel satisfies the integral condition \eqref{1-hormander-2a} for some $\delta>0$ and $1\leq s\leq 2$ with $p<s$. Then $T$ can be extended to a bounded operator from $h^p(\Rn)$ to itself provided that 
		$$
		\max \left\{ \frac{n}{n+\mu}, \, p_0  \right\} < p \leq 1 \ \ \text{where} \ \ \frac{1}{p_0} := \frac{1}{2}+\dfrac{\beta \left( \frac{\delta}{\sigma}+\frac{n}{2} \right)}{n \left( \frac{\delta}{\sigma}-\delta+\beta \right)}
		$$
		and the cancellation condition  \eqref{localCamp} holds.
	\end{theorem}
	
	\begin{proof}
		Let $a$ be a $(p, 2)$ atom in the sense of Definition \ref{goldberg-atom}, supported in $B:=B(x_B,r)$. We will show that $Ta$ is a $(p, s, \lambda,\omega)$ molecule for $\lambda$ satisfying
		$$
		\gamma_p < \frac \lambda s - \frac n{s'} < \min \left\{\mu,\gamma_{p_0}\right\}, \quad \gamma_{p_0}:= n\left(\frac 1{p_0} - 1\right) = -\frac{n}{2}+ \dfrac{\beta \left( \frac{n}{2}+\frac{\delta}{\sigma} \right)}{\beta+\frac{\delta}{\sigma}+\delta}.
		$$
		If $r\geq 1$, conditions (M1) and (M2) will follow by the same arguments presented in the proof of Theorem \ref{continuity-inhomCZO}, provided $\frac \lambda s - \frac n{s'} < \mu$. 
		
		Suppose now that $r<1$.  Analogously to \cite[Lemma 2.1]{AlvarezMilman}, we will actually show some better estimates on $Ta$. In fact, since $1\leq s\leq 2$, from the stronger continuity $L^{q}-L^{2}$ assumption it follows
		$$
		\int_{B} |Ta(x)|^sdx \leq |B|^{1-\frac{s}{2}} \, \| Ta \|_{L^2}^{s} \lesssim |B|^{1-\frac{s}{2}} \| a \|_{L^q}^{s} \lesssim |B|^{1+\frac{s}{q}-s} \| a \|_{L^2}^{s} \lesssim |B|^{1-\frac{s}{p}+s\left(\frac{1}{q}-\frac{1}{2}  \right)} 
		$$
		and so (M1) holds since $1/q-1/2 \geq 0$. To show (M2), consider $0<\rho \leq \sigma \leq 1$, where $\rho$ is a parameter that will be chosen conveniently later. Denote by $2B^{\rho}:= B(x_{B},2r^{\rho})$ and split
		\begin{align*}
			\int_{\Rn} |Ta(x)|^s \, |x-x_B|^{\lambda}dx = \int_{2B^{\rho}} |Ta(x)|^s \, |x-x_B|^{\lambda}dx+\int_{(2B^{\rho})^{c}} |Ta(x)|^s \, |x-x_B|^{\lambda}dx := I_1+I_2.
		\end{align*}
		To estimate $I_1$, we use the $L^{q}-L^{2}$ continuity again and get
		\begin{align*}
			\int_{2B^{\rho}} |Ta(x)|^s \, |x-x_B|^{\lambda}dx \lesssim r^{\lambda \rho} |B^{\rho}|^{1-\frac{s}{2}} \| Ta \|_{L^2}^{s} \lesssim r^{\rho \lambda +n \left(\rho-\frac{s \rho}{2} \right)} \| a \|_{L^q}^{s} \lesssim r^{\rho \lambda +n \left[\rho-\frac{s\rho}{2}+s \left( \frac{1}{q}-\frac{1}{p} \right) \right]} \lesssim r^{\lambda+n\left(1-\frac{s}{p} \right)}
		\end{align*}
		assuming
		\begin{equation} \label{upper-bound-lambda}
			\lambda \leq -n \left( 1-\frac{s}{2} \right)+\frac{ns}{1-\rho}\left( \frac{1}{q}-\frac{1}{2} \right).
		\end{equation}
		Note that this control would not be possible using only the $L^2$-boundedness. For $I_2$, we use \eqref{1-hormander-2a} and then   
		\begin{align} \label{choose-rho}
			\int_{(2B^{\rho})^c}{|Ta(x)|^{s}|x-x_B|^{\lambda}dx} & \leq \sum_{j=0}^{\infty}{ (2^jr^{\rho})^{\lambda} \left\{ \int_{B}{|a(y)| \left[ \int_{A_j(x_B,r^{\rho})}{|K(x,y)-K(x,x_B)|^{s} dx} \right]^{\frac{1}{s}} dy} \right\}^{s}} \nonumber \\
			& \lesssim \ \sum_{j=0}^{\infty}{(2^jr^{\rho})^{\lambda} \,\, \left(|A_j(x_B,r^{\rho})|^{{\frac{1}{s}-1+\frac{\delta}{n} \left(\frac{1}{\rho}-\frac{1}{\sigma} \right)}} \,\, 2^{-\frac{j\delta}{\rho}} \right)^{s} } \,\, \| a \|_{L^2}^{s} \,\, |B|^{\frac{s}{2}} \nonumber \\
			& \lesssim \ \sum_{j=0}^{\infty}{(2^jr^{\rho})^{\lambda} \,\, \left(|A_j(x_B,r^{\rho})|^{{\frac{1}{s}-1+\frac{\delta}{n} \left(\frac{1}{\rho}-\frac{1}{\sigma} \right)}} \,\, 2^{-\frac{j\delta}{\rho}} \right)^{s} \,\, r^{sn\left( 1-\frac{1}{p} \right) }} \nonumber \\
			& = C \  r^{\rho\lambda + n \left[ s+\frac{s\delta}{n}-s\rho \left(1-\frac{1}{s}+\frac{\delta}{n\sigma}\right) -\frac{s}{p} \right]} \sum_{j=0}^{\infty}{2^{j\left[ \lambda-n(s-1)-\frac{s\delta}{\sigma} \right]}} \\
			& \lesssim \ r^{\,\rho\lambda + n \left[ \rho \left(1-\frac{s}{2} \right)+s \left(\frac{1}{q}-\frac{1}{p}\right) \right]} \lesssim r^{\lambda+n\left(1-\frac{s}{p}\right)} \nonumber
		\end{align}			  
		in which we choose $\rho$ to be such that 
		$$ 
		s+\frac{s\delta}{n}-\rho \left(s-1+\frac{s\delta}{n\sigma}\right)= \rho \left( 1-\frac{s}{2} \right)+ \frac{s}{q}, \ \mbox{ i.e. } \ \rho := \frac{n \left( 1-\frac{1}{q}\right)+\delta}{\frac{n}{2}+\frac{\delta}{\sigma}}.
		$$ 
		The convergence of the series in \eqref{choose-rho} follows from \eqref{upper-bound-lambda}, since  by the choice of $\rho$ we have
		$$
		-n \left( 1-\frac{s}{2} \right)+\frac{ns}{1-\rho}\left( \frac{1}{q}-\frac{1}{2} \right) < n(s-1)+s\delta < n(s-1)+\frac{s\delta}{\sigma}.
		$$
		In particular, the restriction on $\lambda$ for this particular choice of $\rho$ is
		$$
		\lambda \leq -n \left( 1-\frac{s}{2} \right) + \dfrac{s\beta \left( \frac{n}{2}+\frac{\delta}{\sigma} \right)}{\beta+\frac{\delta}{\sigma}+\delta}.
		$$
		For the validity of (M3) we proceed in the same way as in the proof of Theorem \ref{continuity-inhomCZO}. Therefore, $Ta$ is a $(p,s,\lambda,\omega)$ molecule provided that $ \displaystyle \max \left\{ \frac{n}{n+\mu}, \, p_{0}  \right\}<p\leq 1$.
	\end{proof}

	
\end{document}